\documentclass[reqno]{amsart}
\usepackage{tkz-euclide}
\usepackage{hyperref}
\usepackage{amsmath}
\usepackage{amssymb} 
\usepackage{mathrsfs}
\usepackage{graphicx}
\usepackage{tikz}
\usepackage[font=footnotesize, labelfont=bf]{caption}
\usepackage{subcaption}
\usepackage{xcolor}  
\usepackage{amsthm}
\usepackage{float}
\usepackage{enumitem}
\usepackage[english]{babel}

\definecolor{LimeGreen}{cmyk}{0.50, 0.5, 1, 0}

\newcommand{\EEE}{\color{black}}

\newcommand{\dx}{\, {\rm d}x}
\newcommand{\dy}{\, {\rm d}y}

\newcommand{\dt}{\, {\rm d}t}
\newcommand{\dP}{\, {\rm d}P}
\newcommand{\e}{\varepsilon}

\DeclareMathOperator{\dist}{dist}
\newcommand{\Sph}{{\mathbb S}}

\newcommand{\ie}{{\it; i.e.}, }

\theoremstyle{plain}
\begingroup
\newtheorem{theorem}{Theorem}[section]
\newtheorem{lemma}[theorem]{Lemma}
\newtheorem{proposition}[theorem]{Proposition}
\newtheorem{corollary}[theorem]{Corollary}
\endgroup

\numberwithin{equation}{section}
\let \O=\Omega
\newcommand{\N}{\mathbb{N}}
\newcommand{\Z}{\mathbb{Z}}
\newcommand{\Q}{\mathbb{Q}}
\newcommand{\R}{\mathbb{R}}

\newcommand{\I}{\mathcal{I}}
\newcommand{\F}{\mathscr{F}}

\newcommand{\Adm}{\mathscr{A}}
\newcommand{\A}{\mathcal{A}}
\renewcommand{\S}{\mathbb{S}}
\renewcommand{\L}{\mathcal{L}}
\renewcommand{\H}{\mathcal{H}}

\newcommand{\dHn}{\, {\rm d}\H^{n-1}}

\newcommand{\loc}{\mathrm{loc}}
\newcommand{\T}{\mathcal{T}}
\newcommand{\x}{\times }
\newcommand{\m}{\mathbf{m}}

\renewcommand{\hom}{\mathrm{hom}}
\newcommand{\defas}{:=}

\newcommand{\wcont}{\subset\subset}

\newcommand{\om}{\omega}
\newcommand{\uu}{{\rm u}}

\newenvironment{step}[1]{\medskip\textit{Step #1:}}{}

\theoremstyle{definition}
\begingroup
\newtheorem{definition}[theorem]{Definition}

\endgroup
\theoremstyle{remark}
\newtheorem{remark}[theorem]{Remark}
\renewcommand{\tilde}{\widetilde}
\newcommand{\sm}{\setminus}

\renewcommand{\d}{\, \mathrm{d}}
\usepackage[normalem]{ulem}

\title[Phase-transition functionals]
{$\Gamma$-convergence and stochastic homogenisation of   phase-transition functionals }

\author[R. Marziani]{Roberta Marziani}
\address[]{Angewandte Mathematik, WWU M\"unster, Germany}
\email[Roberta Marziani]{roberta.marziani@uni-muenster.de}

\begin{document}

%%----------------------------------------------------------------------------------------------------------------------

\begin{abstract} 
	 In  this paper we study
  the asymptotics of singularly perturbed phase-transition functionals of the form
\[
\F_k(u)=\frac{1}{\e_k}\int_A f_k(x,u,\e_k\nabla u)\dx\,,  
\]
where $u \in [0,1]$ is a phase-field variable, $\e_k>0$ a singular-perturbation parameter\ie $\e_k \to 0$, as $k\to +\infty$, and the integrands $f_k$ are such that,
%satisfy mild assumptions. In particular we assume that, 
for every $x$ and every $k$, $f_k(x,\cdot ,0)$ is a double well potential with zeros at 0 and 1.
We prove that the functionals $\F_k$ $\Gamma$-converge (up to subsequences) to a surface functional of the form
\begin{equation*}
\F_\infty(u)=\int_{S_u\cap A}f_\infty(x,\nu_u)\dHn\,,
\end{equation*}
where $u\in BV(A;\{0,1\})$ and $f_\infty$ is characterised by the double limit of suitably scaled minimisation problems.
Afterwards we extend our analysis to the setting of stochastic homogenisation  and prove a $\Gamma$-convergence result for  \emph{stationary random} integrands. 
\end{abstract}

\maketitle

{\small
\noindent 

\medskip

\noindent 
}

\section{Introduction}

The classical theory of  phase transition is based on the  assumption that at the equilibrium two immiscible fluids in a container will separate to create a sharp interface with minimal area.
 Alternatively Cahn and Hilliard in \cite{CH} (see also \cite{WaWalls})  proposed a model  which describes the transition as a  continuous phenomenon concentrated on a thin layer, where a fine mixture of the two fluids is allowed.  
 The connection between the classical theory of phase transition and the Cahn-Hilliard model was conjectured by Gurtin in \cite{Gu} and  rigorously derived, via $\Gamma$-convergence, by Modica in \cite{Mo}, by generalising a previous result together with Mortola \cite{MoMo} (see also \cite{Ste}, and \cite{Alberti}).  \\

 The model considered in \cite{Mo, MoMo} is formalised as follows. 
 Representing by $A\subset\R^n$, open bounded set with Lipschitz boundary, the container of the two fluids, and by $u\colon A\to[0,1]$ the density of 
  the second fluid, the associated Modica-Mortola  functionals are given by 
      \begin{equation}\label{MM-intro}
 	\mathcal {M}_k(u)=
 	\int_A\left(\frac{W(u)}{\e_k}+\e_k^{p-1}|\nabla u|^p\right)\dx\,.
 \end{equation}
  Here $\e_k\searrow0$, as $k\to+\infty$, is an elliptic regularisation parameter which corresponds to the characteristic length-scale of a phase transition. Moreover, $p>1$ and $W\colon\R\to[0,+\infty)$ is a double well potential vanishing at 0 and 1.  
  Clearly the functionals in \eqref{MM} are finite in $W^{1,p}(A)$.
Note that if $(u_k)\subset W^{1,p}(A)$, $0\le u_k\le 1$, is a sequence satisfying $\sup_k\mathcal M_k(u_k)<+\infty$, then from the first term we may deduce that $u_k\to u$ in measure with $u(x)\in\{0,1\}$ a.e. $x\in A$ (which corresponds to a separation of  the two phases). Actually the condition $0\le u_k\le 1$ implies that  $u_k$ converges to $u$ strongly in $L^1(A)$.
 On the other hand the second term penalizes spatial inhomogeneities of  $u_k$ or equivalently the occurrence of too many transition regions. 
According to this, in \cite{Mo} it was shown that the functionals in  \eqref{MM-intro} $\Gamma$-converge to the functional
\begin{equation}\label{lim}
c_p\H^{n-1}(S_u),
\end{equation}
where now $u\in BV(A;\{0,1\})$ (the space of functions with bounded variation taking values in $\{0,1\}$ a.e.).  
Here $S_u$ denotes the set of  discontinuity points of $u$, while the constant $c_p>0$  represents the optimal cost to make a transition between 0 and 1. The quantity in \eqref{lim} can be interpreted as  the surface tension between the two fluids.
We stress also that this result is stable  under the  volume constraint $\int_Au\dx=m$ with $0<m<\L^n(A)$ (that is if we fix the volume of the second fluid).\\

Starting from the result in \cite{MoMo} phase-field models have been extended in several directions including, among others, multi-phase models that describe the behaviour of an arbitrary number of immiscible fluids \cite{Baldo} (see also \cite{Ste, FonTar}), models for heterogeneous fluids that can also undergo a temperature change \cite{Bou}, anisotropic models \cite{BaFo,OwSte}, second order singular perturbations \cite{CDMFL,FonMan, CSZ2011}, and models that allow to analyse the interaction between singular perturbation and homogenisation \cite{AnBraCP, BraZep,CFG,CFHP,Morfe}. These last variants are those which are closest in the spirit to the present work.\\

% Afterwards several variants  of \eqref{MM} were proposed (e.g. \cite{FonTar,Baldo, Bou,OwSte,BaFo,CDMFL,CSZ2011,FonMan, AnBraCP,CFHP,Morfe}).
% Among these it is worth it to mention the multiphase model (i.e., a system with an arbitrary number of  immiscible fluids) treated in \cite{Baldo} (see also \cite{Ste, FonTar}). Another relevant result was obtained in \cite{Bou} were it was analysed a general two-phase model for heterogeneous fluids which possibly undergo to temperature changes. Roughly speaking this means that the two wells are not fixed values (e.g. 0 and 1) but are instead functions of the spatial variable $x$. 
% Furthermore the effect of the interaction between phase transition and homogenisation, was studied by many authors (e.g., \cite{AnBraCP, BraZep, CFHP,Morfe}).
% \\

%The purpose of this paper is twofold: 
Drawing some inspiration from \cite{BMZ}, in the present paper, we aim first at deriving $\Gamma$-convergence and integral representation results for a general class of phase-transition functionals that in some sense unify several of the variants mentioned above. Second, we will apply the abstract $\Gamma$-convergence result to obtain a stochastic homogenisation result for a sequence of phase-field functionals with random stationary integrands. Precisely, in the first part we perform a $\Gamma$-convergence analysis, as $k\to+\infty$, of a family of singularly perturbed phase-transition functionals of the form
\begin{equation}\label{functionals}
	\F_k(u)=\frac{1}{\e_k}\int_A f_k(x,u,\e_k\nabla u)\dx\,,
\end{equation}
where $\e_k\searrow0$ and $u\colon A\to[0,1]$.
We assume that the integrands
$f_k\colon\R^n\times[0,1]\times\R^n\to[0,+\infty)$ belong to a suitable class of functions denoted by $\mathcal F$ (see Section \ref{subs:setting} for its definition). This in particular ensures that for every $k\in\N$ and every $x\in\R^n$
\begin{equation}\label{growth}
c_1\left(W(u)+|\xi|^p\right)\le f_k(x,u,\xi)\le c_2\left(W(u)+|\xi|^p\right),
\end{equation}
for every $u\in[0,1]$ and $\xi\in\R^n$, and for some $0<c_1\le c_2<+\infty$. As a consequence $f_k(x,\cdot,0)$ is a double well potential that vanishes at 0 and 1. In addition the functional $\F_k$ is bounded from above and from below by the Modica-Mortola functional in \eqref{MM-intro}.
          
 Our first main result is contained in Theorem \ref{thm:main-result} which asserts the following: if  the integrands $f_k$ belong to $\mathcal F$, then (up to subsequences) the functionals $\F_k$ $\Gamma$-converge to a surface-type functional of the form 
 \begin{equation}\label{lim2}
 \F_\infty(u)=\int_{S_u\cap A}f_\infty(x,\nu_u)\dHn,
 \end{equation}
          where $u\in BV(A;\{0,1\})$ and $\nu_u$ denotes the external unit normal to $S_u$.  Moreover $f_\infty$  is characterised by a cell formula given by
          \begin{equation}\label{cell-formula}
          	f_\infty(x,\nu)=\limsup_{\rho\to0}\lim_{k\to+\infty}\inf\frac{1}{\e_k}\int_{Q_\rho^\nu(x)}f_k(y,u,\e_k\nabla u)\dx,
          \end{equation}
          where $Q_\rho^\nu(x)$ is a cube centred at $x$ with side-length $\rho$ oriented in the direction $\nu$, and the infimum is taken over all $u\in W^{1,p}(Q_\rho^\nu(x))$, $0\le u\le1$, such that $u=\bar{u}^\nu_{x,\e_k}$ near $\partial Q_\rho^\nu(x)$. Here $\bar{u}^\nu_{x,\e_k}$ is a suitable regularisation (see Section \ref{subs:notation} for the precise definition) of the jump function $u^{\nu}_x$ defined as
          \begin{equation*}
          	u^{\nu}_{x}(y)=\begin{cases}1& \text{if }\; (y-x) \cdot \nu \geq 0\,, 
          		\cr
          		0 & \text{if }\; (y-x) \cdot \nu < 0\,.
          	\end{cases}
          \end{equation*}
          Notice that the presence of ``regularised" boundary conditions is related to the fact that $u$ belongs to $W^{1,p}(A)$ while $u^{\nu}_x$ jumps on the hyperplane $\{(x-y)\cdot\nu=0\}$.

       Due to the generality of our model the proof of Theorem \ref{thm:main-result} relies on an abstract approach known as localisation method (cf. for instance \cite{DM}). This argument consists of the following main steps:  we first localise the functionals $\F_k$ by introducing the dependence on the domain of integration; then, by means of a fundamental estimate,  we show that, up to subsequence, the $\Gamma$-limit exists and admit a surface-integral representation.  Eventually  we identify the surface energy density with the cell formula in \eqref{cell-formula}.\\

In the second part of the paper we are concerned with the case of stochastic homogenisation. We namely consider functionals of type \eqref{functionals} which  depend also on a variable $\omega$ belonging to the set of events $\O$ of a probability space $(\Omega,\mathcal T,P)$. This dependence follows by choosing integrands $f_k$ of type
      \begin{equation}\label{f-stoc}
          f_k(\omega,x,u,\xi)=f\left(
          \omega,\frac{x}{\e_k},u,\xi\right).
      \end{equation}
Here $f$ is a stationary random variable such that $f(\omega,\cdot,\cdot,\cdot)\in\mathcal F$  and is lower semicontinuous in the variables $u$ and $\xi$ for every $\omega\in\O$ (see Section \ref{sect:stochastic-homogenisation},  \ref{meas_f}-\ref{cont_in_xi}). 
% Notice that in the deterministic setting the homogenisation analysis is usually carried out under the assumption of spatial periodicity of the integrands, while in the stochastic setting this assumption is replaced by the more general assumption of stationarity. 
Our second main result is contained in Theorem \ref{thm:stoch_hom_2} which states that if $f_k$ are as in \eqref{f-stoc}, then, almost surely, the functionals $\F_k$ $\Gamma$-converge to a surface functional of type \eqref{lim2}
 whose integrand is homogeneous in the space variable and is given by a cell formula.
\\
\indent
 The proof of Theorem \ref{thm:stoch_hom_2} is achieved by resorting on the abstract $\Gamma$-convergence result Theorem \ref{thm:main-result} as follows. From Theorem \ref{thm:main-result} we first deduce a homogenisation result (cf. Theorem \ref{thm:hom}) in the deterministic setting without requiring any spatial periodicity of the integrands but instead assuming the existence and spatial homogeneity of the limit of a certain cell formula. Afterwards we show that in the above stochastic setting, almost surely this limit indeed exists and is homogeneous almost surely.  This second step  involves the combination of the subadditive ergodic Theorem (see e.g. \cite{AK81, DMM}) with a specific analysis for regularised surface functionals in the spirit of \cite{ACR11, CDMSZ19a}. \\
 \indent
 Eventually to complete our analysis we show that if we restrict to the periodic setting, namely if $f_k$ are of the form $f(\frac{x}{\e_k},u,\xi)$ with $f$ periodic in the space variable, then the corresponding homogenisation result Theorem \ref{thm:det_hom_2} can be deduced  without requiring any lower semicontinuity of $f$  in the variables $u$ and $\xi$ (i.e., without assuming \ref{cont_in_xi}).

 We notice that our stochastic/periodic homogenisation results (Theorem \ref{thm:stoch_hom_2} and Theorem \ref{thm:det_hom_2}) cover in particular the cases considered in \cite{Morfe, CFHP} and the critical regime in \cite{AnBraCP}.
 As a final remark it would be interesting to fully characterise the $\Gamma$-limit for integrands of type \eqref{f-stoc} when the oscillation parameter $\e_k$ is replaced by a second scale parameter $\delta_k\searrow0$. We expect a different behaviour for the two regimes $\e_k\ll\delta_k$ and $\e_k\gg\delta_k$. However this type of analysis goes beyond the purpose of the present paper.
%and will be the object of a forthcoming paper. 
A similar analysis was performed in  \cite{AnBraCP,CFG,CFHP} in the deterministic setting (see also \cite{BEMZ21, BEMZ22} for the corresponding analysis for Ambrosio-Tortorelli type functionals).

  \subsection*{Outline of the paper}
  This paper is organised as follows. In Section \ref{sect:setting} we collect the notation adopted throughout the paper and introduce the functionals we will consider. In Section \ref{sect:main} we state the first main result of the paper, that is, the $\Gamma$-convergence and integral representation result (Theorem \ref{thm:main-result}). We also give
  %a convergence result for some related minimisation problems (Theorem \ref{t:con-min-pb}), and 
  a homogenisation result without any periodicity assumptions (Theorem \ref{thm:hom}). 
  Sections \ref{s:G-convergence}-\ref{sect:prop} contain the proof of Theorem \ref{thm:main-result}). In particular, in Section \ref{s:G-convergence} we implement the localisation method and prove a
  %   as follows:
%   we first prove a fundamental estimate for the functionals $\F_k$ (Proposition \ref{prop:fund-est}) and then use it to deduce a
compactness and integral representation result for the $\Gamma$-limit of $\F_k$ (Theorem \ref{thm:int-rep}); in Section \ref{sect:surface} we characterise the surface integrand (Proposition \ref{p:surface-term}); eventually in Section \ref{sect:prop}  we prove that the limit  surface integrand satisfies some properties (Proposition \ref{prop:f'-f''}).
  %thus fully achieving the proof of the main result, Theorem \ref{thm:main-result} 
  In Section \ref{sect:stochastic-homogenisation}  we state and prove our second main result, namely a stochastic homogenisation result for stationary random integrands (Theorem \ref{thm:stoch_hom_2}). In Section \ref{sect:periodic-hom} we prove a periodic homogenisation result (Theorem \ref{thm:det_hom_2}). Eventually in the Appendix we prove some technical lemmas which are used in Sections \ref{sect:stochastic-homogenisation}  \ref{sect:periodic-hom}.

%%%%%%%%%%%%%%%%%%%%%%%%%%%%%%%%%%%%%%%%%%%%%
%
%    Setting
%
%%%%%%%%%%%%%%%%%%%%%%%%%%%%%%%%%%%%%%%%%%%%%%%%%
\section{Setting of the problem and preliminaries}\label{sect:setting}

\noindent In this section we collect some notation and we introduce the family of functionals we will consider.

\subsection{Notation}\label{subs:notation} We start by listing the notation we will adopt throughout the paper.  

\begin{enumerate}[label=(\alph*)]
\item $n\geq 2$ is a fixed positive integer; 
%we set $\R^m_0:=\R^m\setminus \{0\}$;
\item $\S^{n-1}\defas\{\nu=(\nu_1,\ldots,\nu_n)\in \R^n \colon \nu_1^2+\cdots+\nu_n^2=1\}$ and $\widehat{\S}^{n-1}_\pm\defas\{\nu \in\S^{n-1}\colon \pm\nu_{i(\nu)}>0\}$, where $i(\nu)\defas\max\{i\in\{1,\ldots,n\}\colon \nu_i\neq 0\}$;
\item $\mathcal L^n$ and and $\mathcal H^{n-1}$ denote the Lebesgue measure and the $(n-1)$-dimensional Hausdorff measure on $\R^n$, respectively;
\item for every $A\subset\R^n$ let $\chi_A$ denote the characteristic function of the set $A$;
\item $\A$ denotes the collection of all open and bounded subsets of $\R^n$ with Lipschitz boundary. If $A,B \in \A$ by $A \subset \subset B$ we mean that $A$ is relatively compact in $B$;
\item $Q$ denotes the open unit cube in $\R^n$ with sides parallel to the coordinate axis, centred at the origin; for $x\in \R^n$ and $r>0$ we set $Q_r(x):= rQ+x$. Moreover, $Q'$ denotes the open unit cube in $\R^{n-1}$ with sides parallel to the coordinate axis, centred at the origin, for every $r>0$ we set $Q_r'\defas r Q'$;
\item\label{Rn} for every $\nu\in \Sph^{n-1}$ let $R_\nu$ denote an orthogonal $(n\x n)$-matrix such that $R_\nu e_n=\nu$; we also assume that $R_{-\nu}Q=R_\nu Q$ for every $\nu \in \S^{n-1}$, $R_\nu\in\Q^{n\x n}$ if $\nu\in\S^{n-1}\cap\Q^n$, and that the restrictions of the map $\nu\mapsto R_\nu$ to $\widehat{\Sph}_{\pm}^{n-1}$ are continuous. For an explicit example of a map $\nu \mapsto R_\nu$ satisfying all these properties we refer the reader, \textit{e.g.}, to~\cite[Example A.1]{CDMSZ19};
\item for $x\in\R^n$, $r>0$, and $\nu\in\S^{n-1}$, we define $Q^\nu_r(x):=R_\nu Q_r(x)$. 
%\item for $\xi\in \R^{m \x n}$ we let $u_\xi$ be the linear function whose gradient is equal to $\xi$\ie $u_\xi(x):=\xi x$, for every $x\in \R^n$;
\item\label{jump-fun} for $x\in \R^n$ and $\nu \in \Sph^{n-1}$ we denote with $u_{x}^{\nu}$ the piecewise constant function taking values $0,1$ and jumping across the hyperplane $\Pi^\nu(x):=\{y\in \R^n \colon (y-x) \cdot \nu=0\}$\ie
\begin{equation*}
u^{\nu}_{x}(y):=\begin{cases}1& \text{if }\; (y-x) \cdot \nu \geq 0\,, 
\cr
0 & \text{if }\; (y-x) \cdot \nu < 0\,;
\end{cases}
\end{equation*}
\item\label{1dim-profile} let ${\rm u} \in C^1(\R)$ with $0\leq \uu \leq 1$, be a one-dimensional function such that 
$$\uu(t)=\chi_{(0,+\infty)} \quad\text{for} \quad |t|>1
%,\quad\text{and}\quad u(t)=-u(-t)+1\quad\forall t\in \R
;$$
\item\label{u-bar} for $x\in \R^n$ and $\nu \in \Sph^{n-1}$ we set
\begin{equation*}%\label{uv-bar}
\bar u^\nu_{x} (y):= \uu ((y-x) \cdot \nu) \,;
\end{equation*}
\item\label{u-bar-e} for $x\in\R^n$, $\nu\in\S^{n-1}$ and $\e>0$ we set
\begin{equation*}
\bar u^\nu_{x,\e}(y)\defas \uu \big(\tfrac{1}{\e}(y-x)\cdot\nu)\,.
\end{equation*}
We notice that in particular, $\bar u^\nu_{x,\e}(y)=u^\nu_x$ in $\{y\colon |(y-x)\cdot\nu|>\e\}$, and $\bar{u}_{x,1}^\nu=\bar{u}_x^\nu$;
%and that $\bar u^\nu_{x,\e}(y)=-u^{-\nu}_{x,\e}(y)+1$;
\item for a given topological space $X$, $\mathcal{B}(X)$ denotes the Borel $\sigma$- algebra on $X$. If $X=\R^d$, with $d\in \N$, $d\ge1$ we simply write $\mathcal{B}^d$ in place of $\mathcal B(\R^d)$. For $d=1$ we write $\mathcal B$. 
\end{enumerate}

\medskip

\noindent 
For $A\in \A$ we let $BV(A)$ be the space of functions with bounded variation  (see for instance \cite{AFP} for a detailed exposition of the subject). 
We recall that if $u\in BV(A)$, then its distributional derivative $Du$ is an $\R^n$- valued Radon measure on $A$.  We denote by $|Du|(A)$ its total variation.

Moreover a set $E\subset\R^n$ is a set of finite perimeter in $A$, or a \textit{Caccioppoli set}, if   $\chi_E\in BV(A)$ and we let
\begin{equation*}\label{perimeter}
	\mathcal P_A(E)\defas |D\chi_E|(A)\,,
\end{equation*}
be the \textit{perimeter} of $E$ in $A$. The family of sets with finite perimeter can be identified with the space $BV(A;\{0,1\})$, namely, the space of functions in $BV(A)$ taking values in $\{0,1\}$ almost everywhere. Indeed if $u\in BV(A;\{0,1\})$ then $u=\chi_E$ where $E=\{x\colon u(x)=1\}$ and
 \begin{equation*}\label{c:BV}
Du(B)=\int_{B\cap S_u} \nu_u \d\mathcal H^{n-1}\,,
\end{equation*}
for every $B \in \mathcal B^n$, where $S_u$ is the set of approximate discontinuity points of $u$ which coincides with the reduced boundary of $E$, while $\nu_u$ is the external normal to $S_u$.
Moreover  
%if $E=\{x\colon u(x)=1\}$ then
\begin{equation*}
\mathcal P_A(E)=|Du|(A)=\H^{n-1}(S_u\cap A)\,.
\end{equation*}
\medskip

Throughout the paper $C$ denotes a strictly positive constant which may vary from line to line and within the same expression. 

\subsection{Setting of the problem}\label{subs:setting}
Let $p>1$, let $c_1,c_2$ be given constants such that $0<c_1\leq c_2<+\infty$.
Let $W\colon\R\to[0,+\infty)$ be a double-well potential, that is a continuous function vanishing only at 0 and 1.
%\begin{equation}\label{hyp:zeros-w}
%W(u)=0\quad\text{ if and only if }\quad u\in\{0,1\}\,.
%\end{equation}
%\begin{enumerate}[label=($w\arabic*$)]
%\item(minimum)\label{hyp:zeros-w}  $W(u)=0$ if and only if  $u\in\{0,1\}$;
%\item(lower bound)\label{hyp:growth-w} There exists $0<\alpha$ such that for every $u\in\R$
%\begin{equation*}\label{hyp:growth-W}
%	\alpha(|u|^p-1)\le W(u)
%	%\le \beta(|u|^p+1)
%	\,.
%\end{equation*}
%\end{enumerate}
We denote by $\mathcal{F}:=\mathcal{F}(W,p,c_1,c_2)$ the collection of all functions $f\colon \R^n\x \R\x \R^n\to [0,+\infty)$ satisfying the following conditions:
\begin{enumerate}[label=($f\arabic*$)]
\item(measurability)\label{hyp:meas-f}  $f$ is Borel measurable on $\R^n\x  \R \EEE \x \R^n$;
%\item()\label{hyp:zeros-f}For every $x\in\R^n$
%\begin{equation*}
%f(x,u,0)=0\quad\text{if and only if}\quad u\in\{0,1\}\,;
%\end{equation*}  
\item\label{hyp:lb-f} (lower bound) for every $x \in \R^n$, every $u\in  \R $, and every $\xi\in \R^{n}$
\begin{equation*}
c_1\big( W(u)+|\xi|^p\big) \leq f(x,u,\xi)\, ;
\end{equation*}
\item\label{hyp:ub-f} (upper bound) for every $x \in \R^n$, every $u\in  \R $, and every $\xi \in \R^{n}$
\begin{equation*}
	f(x,u,\xi) \leq c_2( W(u)+|\xi|^p)\, .
\end{equation*}
%
%\item\label{hyp:mon-in-u}(monotonicity in $u$) for every $x\in\R^n$ and every $\xi\in \R^n$, $f(x,\cdot,\xi)$ is decreasing on $(-\infty,0]$ and increasing on $[1,+\infty)$;
%
%\item\label{hyp:cont-f} (continuity in $u$ and $\xi$) \PPP forse non serve!!! \EEE for every $x \in \R^n$ we have
%\begin{equation*}
%|f(x,u_1,\xi_1)-f(x,u_2,\xi_2)| \leq L_1\Big( \big(1+|u_1|^{\gamma-1}+|u_2|^{\gamma-1}\big)|u_1-u_2|+\big(1+|\xi_1|^{p-1}+|\xi_2|^{p-1}\big)|\xi_1-\xi_2|\Big)
%\end{equation*}
%for every $v_1$, $v_2\in  \R$ and every $w_1$, $w_2 \in \R^{n}$;
%\item\label{hyp:min-f} (minimum in $\xi$) for every $x\in\R^n$ and every $u\in  \R$ it holds 
%
%\begin{equation*}
%f(x,u,0)\leq f(x,u,\xi)
%\end{equation*}
%
%for every $\xi\in\R^n$.
\end{enumerate}
\medskip

\noindent
For $k\in \N$ let $(f_k) \subset \mathcal F$ and let $(\e_k)\subset(0,1]$ be a decreasing sequence of  real numbers converging to zero, as $k \to +\infty$.
We consider the sequence of singularly perturbed phase-transitions functionals $\F_k \colon L^1_\loc(\R^n) \times \A \longrightarrow [0,+\infty]$ defined by 
\begin{align}\label{F_e}
\F_k(u, A)\defas
\begin{cases}
\displaystyle\frac{1}{\e_k}\int_A f_k(x,u,\e_k\nabla u)\dx &\text{if}\ u\in  W^{1,p}(A)\, ,0\leq u\leq 1 \, ,\\
+\infty &\text{otherwise}\,. 
\end{cases}
\end{align}
%
%\begin{remark}
%	Thanks to \ref{hyp:mon-in-u} and \ref{hyp:min-f}, the functionals $\F_{k}$ decrease under the transformation $u \to \min\{\max\{u,0\},1\}$. Thus in the definition of $\F_k$ it is not restrictive to assume that it is finite when
%	$u$ satisfies the bounds $0\leq u \leq 1$ and the $\Gamma$-convergence for functionals $\F_k$ defined on functions $u$ with values in $\R$ follows from Theorem \ref{thm:main-result}.
%\end{remark}
\medskip
It is convenient to introduce some further notation.
For  $A\in\A$, $x\in \R^n$ and $\nu\in \Sph^{n-1}$ we consider the following minimisation problem 
\begin{equation}\label{eq:ms}
\m_{k}(\bar u_{x,\e_k}^\nu,A)\defas\inf\{\F_k(u,A)\colon u\in \Adm(\bar{u}_{x,\e_k}^\nu,A)\}\,,
\end{equation}
where 
\begin{equation}\label{c:adm-e}
\Adm(\bar{u}_{x,\e_k}^\nu,A)\defas\big\{u\in W^{1,p}(A),\; 0\leq u\leq 1 \colon \, u=\bar u_{x,\e_k}^\nu\ \text{near}\ \partial A\ \big\}\,.
\end{equation}
%
%We observe that in \eqref{c:adm-e} the boundary datum is prescribed only in 
%\[
%U\cap \{y\in \R^n \colon |(y-x)\cdot\nu|>\e_k\}, 
%\]
%for some neighbourhood $U$ of $\partial Q_\rho^\nu(x)$. 
%\begin{remark}\label{rem:admissible-competitor}
%Clearly, the class of competitors $\Adm_{\e_k,\rho}(x,\nu)$ is nonempty. Indeed, the function $\bar{u}_{x,\e_k}^\nu$ defined as in~\ref{uv-bar-e}, with $\e=\e_k$, satisfies 
%%
%\begin{equation*}
%\bar{u}_{x,\e_k}^\nu=u_x^\nu\; \text{ in }\; \{|(y-x)\cdot\nu|>\e_k\}\,.
%\end{equation*}
%%
%Thus the restriction of $\bar u_{x,\e_k}^\nu$ to $Q_\rho^\nu(x)$ belongs to $\Adm_{\e_k,\rho}(x,\nu)$.
%\end{remark}
In \eqref{eq:ms} by ``$u=\bar u^\nu_{x,\e_k}$ near $\partial A$" we mean that 
the boundary datum is attained in a neighbourhood of  $\partial A$.
Eventually for every $\rho>2\e_k$, every $x \in \R^n$, and every $\nu\in \mathbb S^{n-1}$ we set
\begin{equation}\label{f'}
f'(x,\nu):=\limsup_{\rho\to 0}\frac{1}{\rho^{n-1}}\liminf_{k\to +\infty}\m_{k}(\bar u_{x,\e_k}^\nu,Q^\nu_\rho(x))\, ,
\end{equation}
\begin{equation}\label{f''}
f''(x,\nu):=\limsup_{\rho\to 0}\frac{1}{\rho^{n-1}}\limsup_{k\to +\infty}\m_{k}(\bar u_{x,\e_k}^\nu,Q^\nu_\rho(x))\,.
\end{equation}
\begin{remark}\label{initial-rem}
\begin{enumerate}[label=$(\alph*)$]
\item\label{rem:trivial} 
From \ref{hyp:lb-f} and \ref{hyp:ub-f} together with the assumptions on $W$ we have that
\begin{equation}\label{f-value-at-0}
	f(x,u,\xi)=0\; \text{ if and only if }\; (u,\xi)\in\{(0,0),(1,0)\}\,,
\end{equation}
for every $x\in\R^n$.
As a consequence $f(x,\cdot,0)$ is a double-well potential vanishing at 0 and 1 for every $x\in\R^n$. Moreover the functionals
	$$ \frac{1}{\e_k}\int_Af(x,u,\e_k\nabla u)\dx\,,$$ decrease under the transformation $u \to \min\{\max\{u,0\},1\}$. Thus 
	it is not restrictive to define $\F_k$ to be finite when
	$u\in W^{1,p}(A) $ satisfies the bounds $0\leq u \leq 1$ and the $\Gamma$-convergence for functionals $\F_k$ that are finite when $u\in W^{1,p}(A) $  readily follows from Theorem \ref{thm:main-result}.
\item \label{rem:bounds-Fe}	Assumptions~\ref{hyp:lb-f}--\ref{hyp:ub-f} imply that for every $A\in\A$ and every $u\in  W^{1,p}(A)$ with $0\leq u\leq 1$ it holds
\begin{equation}\label{est:bounds-Fe}
	\begin{split}
		c_1\mathcal M_k(u,A)
		\leq\F_k(u,A)
		\leq c_2\mathcal M_k(u,A)\,;
	\end{split}
\end{equation}
where 
\begin{equation}\label{MM}
	\mathcal M_{k}(u,A)\defas\int_A\bigg(\frac{W(u)}{\e_k}+\e_k^{p-1}|\nabla u|^p\bigg)\dx\,,
\end{equation} 
is the Modica-Mortola functional.
\item \label{compactness}
From the assumption that $\F_k$ is finite when $u\in  W^{1,p}(A)$, $0\le u\le 1$, together with the bound in  \eqref{est:bounds-Fe}  and  the result in \cite{Mo} we easily deduce the following compactness property:
Let $(u_k)\subset W^{1,p}(A)$ be a sequence satisfying $\sup_k \F_k(u_k,A)<+\infty$.  Then $u_k$ converges strongly in $L^1(A)$ to some $u\in BV(A, \{0,1\})$.  \\
If the function $W$ satisfies the following coercivity condition:
\begin{equation*}
W(u)\ge \varphi(|u|)\quad\forall u\in \R\,, 
\end{equation*}
with $\varphi\colon[0,+\infty)\to[0,+\infty)$ such that $\lim_{t\to+\infty}\varphi(t)/t=+\infty$, then the same compactness property holds if we define $\F_k$ to be finite when $u\in  W^{1,p}(A)$ (cf. \cite{Mo}).
\item\label{rem:1dim-energy} 	Let $A\in \A$ be such that $A=A'\x I$ with $A'\subset\R^{n-1}$ open and bounded and $I\subset\R$ open interval.
Let $\nu\in \S^{n-1}$ and set $ A_{\nu}:=R_\nu A$, with $R_\nu$ as in \ref{Rn}. For every $k \in \N$ we have 
\begin{equation}\label{c:needs-name}
	\begin{split}
		\mathcal M_k(\bar{u}_{x,\e_k}^\nu,A_\nu)
		%\int_{A_\nu}\Bigg(\frac{W(\bar{u}_{x,\e_k}^\nu(y))}{\e_k}+\e_k^{p-1}|\nabla\bar{u}_{x,\e_k}^\nu(y)|^p\Bigg)\dy
		\leq \int_{A'}\int_
	\R\big(W(\uu(t))+|\uu'(t)|^p\big)\dt\dy'= C_{\uu}\L^{n-1}(A')\,,
	\end{split}
\end{equation} 
where 
\[
C_{\uu}:=\int_\R\big(W( \uu(t))+|\uu'(t)|^p\big)\dt=\int_{0}^1\big(W( \uu(t))+|\uu'(t)|^p\big)\dt <+\infty\,.
\]
In particular from \ref{hyp:ub-f} and \eqref{c:needs-name} we deduce
\begin{equation}\label{1dim-energy-bis}
	\m_k(\bar u_{x,\e_k}^\nu,A_\nu)\le \F_k(\bar{u}_{x,\e_k}^\nu,A_\nu)\leq c_2C_{\uu}\L^{n-1}(A')\,.
\end{equation}
\end{enumerate}
\end{remark}

%%%%%%%%%%%%%%%%%%%%%%%%%%%%%%%%%%%%%%%%%%%%%%%%%
%
%   Main Results
%
%%%%%%%%%%%%%%%%%%%%%%%%%%%%%%%%%%%%%%%%%%%%%%%%
\section{Statements of the main results}\label{sect:main}
\noindent In this section we collect some of the results of this paper. We start by stating our first main result (Theorem \ref{thm:main-result}), that is a $\Gamma$-convergence   and an integral representation result.
 Next we provide 
% a convergence result for some associated minimisation problems (Theorem \ref{t:con-min-pb}), and
  a homogenisation result without periodicity assumptions (Theorem \ref{thm:hom}). 
The latter, in particular, will be crucial to obtain our second main result in the stochastic setting (see Section \ref{sect:stochastic-homogenisation}).     

\subsection{$\Gamma$-convergence}
The following result shows that, up to subsequences, the functionals $\F_k$  $\Gamma$-converge to a surface integral functional.   
Furthermore, the surface energy density can be characterised  as a double limit of a suitable minimisation problem.

\begin{theorem}[$\Gamma$-convergence]\label{thm:main-result}
Let $(f_k) \subset \mathcal F$ and let $\F_k$ be the functionals as in \eqref{F_e}. Then there exists a subsequence, not relabelled, such that for every $A \in \A$ the functionals $\F_{k}(\cdot\,, A)$ $\Gamma$-converge in  $L^1(A)$ to $\F_\infty(\cdot\,,A)$ with $\F_\infty \colon  L^1_\loc(\R^n) \times \A \longrightarrow [0,+\infty]$ given by
\begin{equation*}\label{F}
\F_\infty(u, A)\defas
\begin{cases}
\displaystyle\int_{S_u \cap A} f_\infty(x,\nu_u)\dHn &\text{if}\ u \in BV(A;\{0,1\})\,,\\[4pt]
+\infty &\text{otherwise}\,,
\end{cases}
\end{equation*}
where $f_\infty \colon \R^n \times \Sph^{n-1} \to [0,+\infty)$ is a Borel function. Moreover for every $x\in\R^n$ and every $\nu\in\Sph^{n-1}$ there hold
\begin{equation*}
		f_\infty(x,\nu)= f'(x,\nu) = f''(x,\nu)\, ,  
\end{equation*}
and
\begin{equation*} \label{prop-f_infty}
%	f_\infty(x,\nu)=f_\infty(x,-\nu)\,,\quad
c_1c_p\le f_\infty(x,\nu)
%=f_\infty(x,-\nu)
\le c_2c_p\,,
\end{equation*}
with $f'$, $f''$ as in \eqref{f'} and \eqref{f''}, respectively and $c_p\defas p(p-1)^{\frac{1-p}{p}}\int_0^1W(t)^\frac{p-1}{p}\dt$. 
\end{theorem}
For the reader's convenience we divide the proof of Theorem \ref{thm:main-result} in  three parts which can be found in Sections \ref{s:G-convergence}, \ref{sect:surface}, and \ref{sect:prop},  respectively. Precisely, in Section \ref{s:G-convergence} we show 
 that there is a sequence $(k_j)$, with $k_j \to +\infty$ as $j\to +\infty$, such that for every $A\in \mathcal A$ the corresponding functionals  $\F_{k_j}(\, \cdot\,, A)$ $\Gamma$-converge to a functional which is finite in $BV(A;\{0,1\})$ and is of the form 
\begin{equation*}\label{F-hat}
	\int_{S_u\cap A} \hat f(x,\nu_u)\dHn\,, 
\end{equation*}
for some Borel function $\hat f$ (see Theorem \ref{thm:int-rep}). In Section  \ref{sect:surface}  we identify $\hat f$ by showing that it is equal both to $f'$ and $f''$ which in particular are shown to coincide along the $\Gamma$-convergent subsequence (see Proposition \ref{p:surface-term}). Eventually in Section \ref{sect:prop} we prove that the function $f'$ (respectively $f''$) is  Borel measurable and satisfies $c_1c_p\le f'(x,\nu)
%=f'(x,-\nu)
\le c_2c_p$ (respectively $c_1c_p\le f''(x,\nu)
%=f''(x,-\nu)
\le c_2c_p$) for every $x\in\R^n$ and every $\nu\in \S^{n-1}$ (see Proposition \ref{prop:f'-f''}).

\medskip

By Theorem \ref{thm:main-result} and  the Urysohn property of $\Gamma$-convergence  \cite[Proposition 8.3]{DM} we can deduce the following useful property. 

\begin{corollary}\label{c:cor-main-thm}
Let $(f_k) \subset  \mathcal F$ and let $\F_k$ be the functionals as in \eqref{F_e}. Let $f'$, $f''$ be as in \eqref{f'} and \eqref{f''}, respectively.
Assume that  
\[
f'(x,\nu) = f''(x,\nu)=:f_\infty(x,\nu),\; \text{ for every $x\in\R^n$ and every $\nu\in\Sph^{n-1}$},  
\]
for some Borel function $f_\infty \colon \R^n \times \Sph^{n-1} \to [0,+\infty)$.
Then, for every $A \in \A$ the functionals $\F_{k}(\cdot\,, A)$ $\Gamma$-converge in $L^1(A)$ to $\F_\infty(\cdot\,,A)$ with $\F_\infty \colon L^1_\loc(\R^n) \times \A \longrightarrow [0,+\infty]$ given by
\[
\F_\infty(u,A)\defas
\begin{cases}
\displaystyle\int_{S_u \cap A} f_\infty(x,\nu_u)\dHn &\text{if}\ u \in BV(A;\{0,1\})\, ,
\\[4pt]
+\infty &\text{otherwise}\,.
\end{cases}
\]
\end{corollary}
\subsection{Homogenisation} 
In this subsection we prove a general homogenisation theorem without assuming any spatial periodicity of the integrands $f_k$. This theorem will be employed to prove the stochastic homogenisation result Theorem \ref{thm:stoch_hom_2}.
 
\medskip 
 
We need to introduce some further notation. 
We fix $f\in\mathcal{F}$,  $A\in \A$, $u\in W^{1,p}(A)$,  $z\in \R^n$, $\nu \in \Sph^{n-1}$, and let $\bar u^\nu_z$ be as in \ref{u-bar} (with $x=z$). Then we set
\begin{equation}\label{def:F-1}
\F(u,A)\defas\int_Af(x,u,\nabla u)\dx\,;
\end{equation}
and
\begin{equation}\label{eq:m-bis}
\m(\bar u^\nu_z,A):=\inf\Bigl\{\F(u,A)\colon u\in \Adm(\bar u^\nu_z,A)\Bigr\}\,,
\end{equation}
where $\Adm(\bar u^\nu_z,A)$ is as in \eqref{c:adm-e}, with $\bar u^\nu_{x,\e_k}$ replaced by $\bar u^\nu_z$ (that is, $x=z$ and $\e_k=1$).

When the integrands $f_k$ in the definition of $\F_k$ are of type
\begin{equation}\label{eq:f-osc}
f_k(x,u,\xi):= f\Big(\frac{x}{\e_k},u,\xi\Big)\,,
\end{equation}
we can deduce the following homogenisation result.
%We stress again that we will not assume any spatial periodicity of $f$. 
%
\begin{theorem}[Deterministic homogenisation]\label{thm:hom} Let $f\in \mathcal F$ and let $\m(\bar u_{rx}^\nu,Q^\nu_r(rx))$ be as in \eqref{eq:m-bis} with $z=rx$ and $A=Q_r^\nu(rx)$.  Assume that for every $x\in \R^n$, $\nu\in \S^{n-1}$ the following limit
\begin{equation}\label{eq:f-hom}
\lim_{r\to +\infty} \frac{\m(\bar{u}^\nu_{rx},Q^\nu_r(rx))}{r^{n-1}}=:f_{\rm hom}(\nu)\,,
\end{equation}
exists and is independent of $x$. Then, for every $A\in \A$ the functionals $\F_k(\cdot, A)$ defined in \eqref{F_e} with $f_k$ as in \eqref{eq:f-osc} $\Gamma$-converge in $ L^1(A)$ to the functional $\F_{\rm hom}(\cdot,A)$, with $\F_{\rm hom}\colon L^1_{\loc}(\R^n)\times \A \longrightarrow [0,+\infty]$ given by 
\begin{equation*}%\label{eq:F-hom}
\F_{\rm hom}(u,A)\defas
\begin{cases}
\displaystyle\int_{S_u \cap A} f_{\rm hom}(\nu_u)\dHn &\text{if}\ u \in SBV^p(A)\, ,u\in\{0,1\}\text{ a.e. }
\\[4pt]
+\infty &\text{otherwise}\,.
\end{cases}
\end{equation*}
\end{theorem}
\begin{proof}
Let $f'$, $f''$ be as in \eqref{f'}, \eqref {f''}, respectively. By Corollary \ref{c:cor-main-thm} it suffices to show that 
	\begin{equation}\label{c:id-hom}
		f_{\hom}(\nu)=f'(x,\nu)=f''(x,\nu)\,,
	\end{equation}
for every $x\in\R^n$ and $\nu\in \Sph^{n-1}$.

Let $x\in\R^n$, $\nu\in \S^{n-1}$, $k\in\N$  and $\rho>2\e_k$ be fixed and let $u\in \Adm(\bar{u}_{x,\e_k}^\nu,Q_\rho^\nu(x))$. 
	 Define $u_k\in  W^{1,p}(Q^\nu_{\frac{\rho}{\e_k}}(\frac{x}{\e_k}))$  as $u_k(y):=u(\e_ky)$. Then clearly  $u_k \in \Adm(\bar{u}_{\frac{x}{\e_k}}^\nu,Q^\nu_{\frac{\rho}{\e_k}}(\tfrac{x}{\e_k}))$ and a change of variables gives
$$
\F_{k}(u,Q^\nu_\rho(x))=\e_k^{n-1}\int_{Q^\nu_{\frac{\rho}{\e_k}}(\frac{x}{\e_k})}f(x,v_k,\nabla v_k)\dx
\,.
$$
Hence by setting $r_k:=\frac{\rho}{\e_k}$ we get
	\begin{equation*}
\m_{k}(\bar{u}_{x,\e_k}^\nu,Q_\rho^\nu(x))=\e_k^{n-1}\m\big(\bar{u}_{\frac{x}{\e_k}}^\nu,Q_{\frac{\rho}{\e_k}}^\nu\big(\tfrac{x}{\e_k}\big)\big)=
\frac{\rho^{n-1}}{r_k^{n-1}}\m\big(\bar u^\nu_{r_k\frac{x}{\rho}},Q_{r_k}\big(r_k\tfrac{x}{\rho}\big)\big)\,.
\end{equation*}
Finally using \eqref{eq:f-hom} with $x/\rho$ in place of $x$ we obtain
\begin{equation*}
\lim_{k\to+\infty}\frac{\m_{k}(\bar u_{x,\e_k}^\nu,Q_\rho^\nu(x))}{\rho^{n-1}}=\lim_{k\to+\infty}\frac{\m\big(\bar u^\nu_{r_k\frac{x}{\rho}},Q_{r_k}^\nu\big(r_k\tfrac{x}{\rho}\big)\big)}{r_k^{n-1}}=f_{\hom}(\nu)\,,
\end{equation*}
and therefore  $f'(x,\nu)=f''(x,\nu)=f_{\hom}(\nu)$, for every $x\in \R^n$, $\nu\in \Sph^{n-1}$.
\end{proof}

%%%%%%%%%%%%%%%%%%%%%%%%%%%%%%%%%%%%%%%%%
%
%Integral Representation
%
%%%%%%%%%%%%%%%%%%%%%%%%%%%%%%%%%%%%%%%%%%%

\section{$\Gamma$-convergence and integral representation}\label{s:G-convergence}

\noindent In this section we prove that, up to subsequences, the functionals $\F_k$ $\Gamma$-converge to a surface integral functional.
In order to do that we apply the so called localisation method of $\Gamma$-convergence.
In particular we closely follow the theory described in \cite[Chapters 14-18]{DM} and  \cite[Chapters 10, 11]{BDf}, for this reason we discuss here only the main adaptations to our case.
% through a standard argument which combines the localisation method of $\Gamma$-convergence (see \textit{e.g.}, \cite[Chapters 14-18]{DM} or \cite[Chapters 10, 11]{BDf}) together with an integral-representation result in $BV(A;\{0,1\})$  \cite[Theorem 3]{BFLM02}. 
%For this reason here we only 
%detail the adaptations of the theory to our specific setting, while we refer the reader to the literature for the more standard aspects.    

We start by proving a preliminary result, namely we show that the functionals $\F_k$ satisfy a fundamental estimate, uniformly in $k$ similarly to \cite[Lemma 3.2]{AnBraCP}. 
\begin{proposition}[Fundamental estimate]\label{prop:fund-est}  Let  $\F_k$ be as in~\eqref{F_e}. Then there exist $K>0$ depending only on $c_1,c_2,p$ and $W$ and functions $\omega_{k}\colon L^1_{\loc}(\R^n)^2\times \A^3\to[0,+\infty)$ $($depending only on $k)$ such that for every $A,\,A',\,B\in \A$ with $A \subset\subset A'$ the following hold:
\begin{enumerate}[label=$(\roman*)$]
\item\label{fund-i} For every  $u\in W^{1,p}(A')$, $v\in W^{1,p}(B)$, $0\leq u,v\leq 1$, there exists $w\in W^{1,p}(A\cup B)$ with $0\leq w\leq 1$ such that
\begin{equation*}
w=u\quad\text{ a.e.\ in }\quad A\quad\text{ and }\quad
w=v\quad\text{ a.e.\ in }\quad B\sm\overline{A}'\,,
\end{equation*}
and
\begin{align*}\label{fund-est}
	\F_k(w, A\cup B)\leq(1+K\e_k) (\F_k(u,A')+ \F_k(v, B))+\omega_{k}(u,v,A,A',B)\,;
\end{align*}
\item\label{fund-ii} For every  $u\in W^{1,p}(A')$, $v\in W^{1,p}(B)$ $0\leq u,v\leq 1$ and every  $\tilde S\subset (A'\setminus\overline A)\cap B$ with
	\begin{equation*}
	u(x)=v(x)\in\{0,1\}\quad \text{ a.e. in } \quad\tilde S \,;
\end{equation*}
we have
\begin{equation*}
	\begin{split}
		\omega_{k}(u,v,A,A',B)\le K
	 \L^n((A'\setminus\overline A)\cap B\setminus\tilde S)\,;
	\end{split}
\end{equation*}
\item\label{fund-iii} For 
%every $\e_k\searrow0^+$, for
 every $(u_k),(v_k)\subset L^1_{\rm loc}(\R^n)$ having the same limit as $k\to\infty$ in $L^1((A'\setminus \overline A)\cap B)$ and satisfying
\begin{equation*}\label{eq:fund-estimate-ub}
	\sup_{k\in\N}\left(\F_k(u_k,A')+\F_k(v_k,B)\right)<+\infty\,,
\end{equation*}
it holds 
\begin{equation}\label{resto}
	\lim_{k\to\infty}\omega_{k}(u_k, v_k,A,A',B)=0\,.
\end{equation}
\end{enumerate}
\end{proposition}
%
%\begin{remark}
%Note that the presence of  $\eta$ in \eqref{fund-est} plays a role only when $\e_k=1$ and in particular it applies to the functional $\F$ defined in \eqref{def:F-1}. 
%This will be useful to prove some technical Lemmas in Section \ref{appendix} employed to prove Theorem \ref{thm:stoch_hom_2}.
%If  instead $\e_k\searrow0^+$ as $k\to\infty$ it is sufficient to apply Proposition \ref{prop:fund-est} with $\eta=1$.
%\end{remark}
\begin{proof}
Let $A,\,A',\,B\in \A$ with $A\wcont A'$ be fixed. 

\step 1 we show \ref{fund-i}.
Let $d\defas\dist(A;\R^n\setminus A')$,
$N_k\defas\lfloor 2\e_k^{-1}\rfloor$ (here $\lfloor a\rfloor$  denotes the integer part of $a$), and let $A_1,\ldots,A_{N_k+1}\in \A$ with
\begin{equation*}
A\wcont A_1\wcont \dots \wcont A_{N_k+1} \wcont A'\,,
\end{equation*}
and 
\begin{equation*}
\dist(A_i;\R^n\setminus A_{i+1})\ge\frac{d}{N_k+2}\quad i=1,\dots,N_k\,.
\end{equation*}
For each $i=1,\ldots,N_k$ let $\varphi_i$ be a smooth cut-off function between $A_{i}$ and $A_{i+1}$ such that
\begin{equation*}
M_k\defas\max_{1\leq i\leq N_k} \|\nabla \varphi_i\|_\infty\le\frac{2(N_k+2)}{d}\,.
\end{equation*}
Let $u$ and $v$ be as in the statement.
For $i=1,\ldots,N_k$ we define $w^i \in W^{1,p}(A\cup B)$ as follows
\begin{equation}\label{w^i}
w^i := \varphi_{i} u + (1-\varphi_{i}) v\,.
\end{equation}
By setting  $S_i\defas (A_{i+1}\sm\overline{A}_{i})\cap B$ we obtain 
\begin{equation}\label{4est2}
\begin{split}
\F_k (w^i, A\cup B)&\leq \F_k (u, A_{i})  
+ \F_k (v, B\setminus \overline A_{i+1})+\F_k (w^i, S_i)\\
&\leq \F_k (u, A')  
+ \F_k (v, B)+\F_k (w^i, S_i)
\,.
\end{split}
\end{equation}
Now from~\ref{hyp:ub-f} we deduce
\begin{equation}\label{est:fund-est-Fs1}
	\begin{split}
	\F_k(w^i,S_i)\leq c_2\int_{S_i\cap B}\hspace*{-.5em}\bigg(\frac{W(w^i)}{\e_k}+\e_k^{p-1}|\nabla w^i|^p\bigg)\dx\,.
	\end{split}
\end{equation}
Moreover a convexity argument gives
\begin{equation}\label{prop:w}
	\begin{split}
	|\nabla w_i|^p&\le 2^{p-1}\left(|\nabla u|^p+|\nabla v|^p+|\nabla\varphi_i|^p|u-v|^p\right)\\
	&\le 2^{p-1}\left(|\nabla u|^p+|\nabla v|^p+M_k^p|u-v|^p\right)
	\,.
	\end{split}
\end{equation}
From \eqref{est:fund-est-Fs1}, \eqref{prop:w} and~\ref{hyp:lb-f} we obtain
\begin{equation}\label{est:fund-est-Fs2}
	\begin{split}
\F_k(w^i,S_i)&\leq \frac{2^{p-1}c_2}{c_1}\big(\F_k(u,S_i)+\F_k(v,S_i)\big)\\
&
+ 2^{p-1}c_2M_k^p\e_k^{p-1}\int_{S_i}|u-v|^p\dx
+c_2\int_{S_i}\frac{W(w^i)}{\e_k}\dx\,.
	\end{split}
\end{equation}
Now summing up in~\eqref{4est2} over all $i$,  by averaging we find an index $i^*\in\{1,\ldots,N_k\}$ such that
\begin{align}\label{summing}
	\F_k(w^{i^*},A\cup B) \leq\frac{1}{N_k} \sum_{i=1}^{N_k}\F_k (w^i, A\cup B)
	\leq \F_k(u,A')+\F_k(v,B)+ \frac{1}{N_k} \sum_{i=1}^{N_k}\F_k (w^i, S_i)\,.
\end{align}
By~\eqref{est:fund-est-Fs2} we have
\begin{equation}\label{omega_k}
 \frac{1}{N_k}\sum_{i=1}^{N_k}\F_k (w^i, S_i)\le\frac{\widehat M}{N_k}\left(\F_k(u,A')+\F_k(v,B)+M_k^p\e_k^{p-1}\int_{S}|u-v|^p\dx+\sum_{i=1}^{N_k}\int_{S_i}\frac{W(w^i)}{\e_k}\dx
\right)\,,
\end{equation}
with $\widehat M\defas\max\{\frac{2^{p-1} c_2}{c_1},2^{p-1}c_2,c_2\}$ and $S\defas (A'\setminus  \overline A)\cap B$.  Using that  $\frac{2-\e_k}{\e_k}\le N_k\le \frac2{\e_k}$ we have 
\begin{equation*}\label{constants}
\frac{\widehat M}{N_k}\le \frac{\widehat M \e_k}{2-\e_k}\le \widehat M\e_k\,,
\quad\text{ and }\quad
	\frac{\widehat M}{N_k}M_k^p\e_k^{p-1}\le
	C\e_k^p\left(\frac{1}{\e_k}+1\right)^p\le C
	\,.
\end{equation*}
This together with \eqref{summing} and \eqref{omega_k} give
\begin{equation*}
\F_k (w^i, A\cup B)\le
 (1+\widehat M\e_k) \left(\F_k(u,A')+\F_k(v,B)\right)+
 \omega_{k}(u,v,A,A',B)
\end{equation*}
where
\begin{equation}\label{def:omega}
\omega_{k}(u,v,A,A',B)\defas  C\left(\int_S|u-v|^p\dx+ \sum_{i=1}^{N_k}\int_{S_i}W(w^i)\dx\right)
\,,
\end{equation}
for some positive constant $C>0$ independent of $k$.

\step 2 we show \ref{fund-ii}.  Assume that $0\le u,v\le 1$ are such that 
\begin{equation*}
	u(x)=v(x)\in\{0,1\}\quad \text{ a.e. in } \quad\tilde S\subset S\,.
\end{equation*}
Then clearly $\int_{\tilde S}|u-v|^p\dx=0$ and
\begin{equation*}
W(w^i)=0 \quad \text{ a.e.\ in }\quad S_i\cap\widetilde S\,.
\end{equation*}
Moreover, since $0\le w^i\le1$ for all $i=1,...,N_k$, by the continuity of  $W$ 
\begin{equation*}
W(w_i)\le M \quad \text{ a.e.\ in }\quad S_i\setminus \widetilde S\,, \ \forall i=1,...,N_k\,.
\end{equation*}
Hence \eqref{def:omega} readily implies
\begin{equation*}
\omega_{k}(u,v,A,A',B)\le C(2^p+M )\L^n(S\setminus \tilde S)\,,
\end{equation*}
and  by setting $K\defas \max\{\widehat M,C(2^p +M)\} $ we deduce \eqref{fund-ii}.

\step 3 we show \ref{fund-iii}. Let $(u_k),(v_k)\subset L^1_{\rm loc}(\R^n)$, $0\le u_k,v_k\le1$, be two sequences converging to $w$ in $L^1(S)$ and such that
\begin{equation*}
	\sup_{k\in\N}\left(\F_k(u_k,A')+\F_k(v_k,B)\right)<C\,.
\end{equation*}
Since $0\le u_k,v_k\le1$ we immediately deduce that $u_k$ and $v_k$ converge to $w$ also in $L^p(S)$.
%We start by showing that $u_k$ and $v_k$ converge to $w$ also in $L^p(S)$. Indeed by \eqref{hyp:growth-W} there is $r>0$ satisfying
%\begin{equation*}
%W(z)\ge \frac{\alpha}{2}|z|^p\quad\text{if}\quad |z|>r\,.
%\end{equation*}
%Then by defining
%\begin{equation*}
%	w^r(x)\defas-r\vee(r\wedge w(x))\quad x\in\R^n
%\end{equation*}
%and $u_k^r$ and $v_k^r$ accordingly to $u_k$ and $v_k$ respectively we have that $u_k^r$ and $v_k^r$ converge to $w^r$ in $L^p(S)$; moreover 
%\begin{equation*}
%\begin{split}
%\int_S|u_k-u_k^r|^p\dx\le2 \int_{\{|u_k|>r\}}|u_k|^p\dx
%\le \frac{4}{\alpha}\int_SW(u_k)\dx\le \frac{4M\e_k}{\alpha}\,,
%\end{split}
%\end{equation*}
%and the same by replacing $u_k$ and $u_k^r$ with $v_k$ and $v_k^r$ respectively. Hence we can conclude that $u_k$ and $v_k$ converge to $w^r$ in $L^p(S)$ and as they converge to $w$ in $L^1(S)$ we deduce $w^r=w$.

Thus \eqref{resto} follows if we show that 
\begin{equation*}
\lim_{k\to\infty}\sum_{i=1}^{N_k}\int_{S_i}W(w_k^i)\dx=0\,,
\end{equation*}
with $w^i_k$ defined as in \eqref{w^i} with $u_k$ and $v_k$ in place of $u$ and $v$ respectively. 
Let 
\begin{equation*}
w_k\defas\begin{cases}
w^i_k&\text{ if } x\in S_i\,, i=1,\dots N_k\,,\\
w&\text{ otherwise}\,;
\end{cases}
\end{equation*}
then clearly $w_k$ converges to $w$ in $L^p(S)$, moreover since $W$ is continuous 
\begin{equation*}
0\le \lim_{k\to\infty}\sum_{i=1}^{N_k}\int_{S_i}W(w_k^i)\dx\le \lim_{k\to\infty}\int_SW(w_k)\dx=
\lim_{k\to\infty}\int_SW(u_k)\dx\le \lim_{k\to\infty}\e_k C=0\,.
\end{equation*}

\end{proof}
%
%\begin{remark}\label{rem:fund-estimate} For later convenience we observe that by \eqref{def:omega}, and \ref{hyp:growth-W} it follows that  if 
%	\begin{equation*}
%u(x)=v(x)\in\{0,1\}\quad \text{ a.e. in } \quad\tilde S\subset (A'\setminus\overline A)\cap B\,;
%	\end{equation*}
%	then
%	\begin{equation*}
%		\begin{split}
%\omega_k(u,v,A,A',B)\le o_k(1)\left(\F_k(u,A)+\F_k(v,B)
%\right)+C\left(2^p+ \frac{2\beta}{\e_k}\right)\L^n((A'\setminus\overline A)\cap B\setminus\tilde S)\,.
%		\end{split}
%	\end{equation*}
%%where $S\defas (A'\setminus\overline A)\cap B$, $S_1\defas S\cap\{u=v\}$, and $S_2\defas S\setminus S_1$.
%	\end{remark}
With the help of Proposition \ref{prop:fund-est}, we can now prove the following $\Gamma$-convergence result.  
\begin{theorem}\label{thm:int-rep}
Let $\F_k$ be as in~\eqref{F_e}. Then there exist a subsequence $(\F_{k_j})$ of $(\F_k)$ and a functional $\F_\infty\colon L^1_\loc(\R^n)\times\A\longrightarrow [0,+\infty]$ such that for every $A\in\A$ the functionals $\F_{k_j}(\cdot\, ,A)$ $\Gamma$-converge in $ L^1(A)$ to $\F_\infty(\cdot\,,A)$. Moreover, $\F_\infty$ is given by
\begin{equation*}%\label{eq:glimit-ir}
\F_\infty(u,A)\defas
\begin{cases}
\displaystyle\int_{S_u\cap A} \hat f(x,\nu_u)\dHn\, &\text{if}\ u \in BV(A)\, ,\ v\in\{0,1\}\, \text{a.e.\ in } A\, ,\\
+\infty &\text{otherwise}\,,
\end{cases}
\end{equation*}
with  ${\hat f}\colon \R^n\x\S^{n-1}\to[0,+\infty)$ defined as
\begin{equation}\label{eq:f-hat}
 \hat f(x,\nu):=\limsup_{\rho\to 0}\frac{1}{\rho^{n-1}}\m(u_{x}^\nu,Q_\rho^\nu(x))\,,
\end{equation}
for every $x\in \R^n$  and $\nu\in \Sph^{n-1}$,
where for $A\in\A$ and $\bar{u}\in BV(A,\{0,1\})$
\begin{align*}\label{def:m}
\m(\bar{u},A)\defas\inf\{\F_\infty(u,A)\colon u\in BV(A;\{0,1\})\, ,\ u=\bar{u}\ \text{near}\ \partial A\}\,.
\end{align*}
\end{theorem}
\begin{proof}  The proof is rather standard, thus here we only sketch it and refer to \cite{DM} for more details. 
By \cite[Theorem 16.9]{DM} there exists a subsequence $(k_j)$, with $k_j \to +\infty$ as $j \to +\infty$, such that the inner regular envelopes of the $\Gamma$-lower limit and of the $\Gamma$-upper limit coincide. More precisely for every $u\in  L^1_\loc(\R^n)$ and every $A\in \A$
	\begin{equation}\label{eq:Gamma-bar}
		\sup_{A'\subset \subset A,\, A'\in \A}\F'(u,A')\quad =\sup_{A'\subset \subset A,\, A'\in \A}\F''(u,A')=:\F_\infty(u,A)\,,
	\end{equation}
 with $\F',\F'' \colon  L^1_\loc(\R^n)\times \A \longrightarrow [0,+\infty]$ defined as
 \begin{equation*}
 	\F'(\,\cdot\,,A) \defas \Gamma\hbox{-}\liminf_{k_j \to +\infty}\F_{k_j}(\,\cdot\,,A)\quad \text{and}\quad
 	\F''(\,\cdot\,,A) \defas \Gamma\hbox{-}\limsup_{k_j \to +\infty}\F_{k_j}(\,\cdot\,,A)\,.
 \end{equation*}
Now  by recalling Remark \ref{rem:bounds-Fe} and \cite[Theorem 1]{MoMo} we can find a constant $C>0$ such that
	\begin{equation}\label{est:bounds-F'-F''}
		\begin{split}
			\frac{1}{C}\H^{n-1}(S_u\cap A)\leq\F'(u,A)\leq\F''(u,A)\leq C\H^{n-1}(S_u\cap A)\,,
		\end{split}
	\end{equation}
	for every $A\in\A$ and every $u\in BV(A;\{0,1\})$ and
	\begin{equation}\label{eq:out-domain}
		\F'(u,A)=\F''(u,A)= +\infty \quad \text{if  } \; u\notin BV(A;\{0,1\}) \,. 
	\end{equation}
In particular, \eqref{est:bounds-F'-F''} together with \cite[Proposition 18.6]{DM} and Proposition~\ref{prop:fund-est} imply that 
	\begin{equation*}
		\F_\infty(u,A)=\F'(u,A)=\F''(u,A) \quad \text{if}\quad u\in BV(A;\{0,1\})\,,
	\end{equation*}
	while \eqref{eq:Gamma-bar} and \eqref{eq:out-domain} yield
	\begin{equation*}
		\F_\infty(u,A)=\F'(u,A)=\F''(u,A)=+\infty \quad \text{if  } \; u\notin BV(A;\{0,1\}) \,. 
	\end{equation*}
Hence, $\F_\infty(\cdot, A)$ coincides with the $\Gamma$-limit of $\F_{k_j}(\cdot,A)$ on $L^1(\R^n)$, for every $A\in \A$.\\
	
We now observe that  the functional $\F_\infty$ defined in \eqref{eq:Gamma-bar} satisfies the following properties:  the functional $\F_\infty(\cdot,A)$ is $L^1(\R^n)$ lower semicontinuous \cite[Propositions 6.7--6.8 and Remark 15.10]{DM} and local \cite[Remark 15.25 and Proposition 16.15]{DM}, while the set function $\F_\infty(u,\cdot)$ is inner regular, increasing and superadditive \cite[Propositions 6.7, 16.12 and Remark 15.10]{DM}. Finally by means of the fundamental estimate Proposition \ref{prop:fund-est} we can appeal to \cite[Proposition 18.4]{DM} to deduce that $\F_\infty(u,\cdot)$ is also a subadditive set function. 
	This together with the measure-property criterion of De Giorgi and Letta (see \textit{e.g.}, \cite[Theorem 14.23]{DM}) imply that for every $u\in L^1_\loc(\R^n)$ the set function $\F_\infty(u,\cdot)$ is the restriction to $\A$ of a Borel measure.   \\
Thus we can invoke  \cite[Theorem 3.1]{AmBra} (see also \cite[Theorem 3]{BFLM02}) and conclude that that for every $A\in \A$ and $u\in BV(A;\{0,1\})$ we can represent the $\Gamma$-limit $\F_\infty$ can be represented in an integral form as  
\begin{equation*}
\F_\infty(u,A)=
\int_{S_u\cap A}\hat f(x,\nu_u)\dHn\,,
\end{equation*}
with $ \hat f$ given by~\eqref{eq:f-hat}.
\end{proof}

%%%%%%%%%%%%%%%%%%%
%SECTION: SURFACE
%%%%%%%%%%%%%%%%%%

 \section{Identification of the surface integrand}\label{sect:surface} 
\noindent In this section we identify the surface integrand $ \hat f$. Namely, we show that $\hat f$ coincides with both $f'$ and $f''$, defined as in \eqref{f'} and \eqref{f''}  with respect to the $\Gamma$-converging subsequence. 

\medskip
\begin{proposition}\label{p:surface-term}
	Let  $(f_k) \subset \mathcal F$ ; let $(k_j)$ and $\hat f$ be as in Theorem~\ref{thm:int-rep}. Then, it holds 
	\begin{align*}
	\hat f(x,\nu)=f'(x,\nu) = f''(x,\nu)\,,
	\end{align*}
	for every $x\in\R^n$ and $\nu\in\Sph^{n-1}$, where $f'$ and $f''$ are, respectively, as in \eqref{f'} and \eqref{f''}, with $k$ replaced by $k_j$.
\end{proposition}
\begin{proof}

For notational simplicity, in what follows we still denote with $k$ the index of the subsequence provided by Theorem~\ref{thm:int-rep}.

\medskip
By definition we have  $f'\le f''$, hence to conclude we need to show that $\hat f\le f'$ and $\hat f\ge f''$.  We divide the proof into two steps.

\medskip

\emph{Step 1:} In this step we show that $\hat f(x,\nu)\le f'(x,\nu)$ for every $x\in\R^n$ and $\nu\in\Sph^{n-1}$.

Let $\rho>0$ and $\eta>0$ be fixed, and for every $k$ choose $u_k \in\Adm(\bar{u}_{x,\e_k}^\nu,A)$ satisfying
\begin{align}
	\F_{k}(u_k,Q_\rho^\nu(x)) &\leq \m_k(\bar u_{x,\e_k}^\nu,Q_\rho^\nu(x)) +\eta\rho^{n-1}\label{eq:q-min-e}\\
		&\le (c_2 C_{\uu}+\eta)\rho^{n-1}\,.\label{uniformbound:me}
\end{align}
Here the last inequality follows from~\eqref{1dim-energy-bis}.  
% by definition of $\m_{k}( \bar u_{x,\e_k}^\nu,Q_\rho^\nu(x))$ there exists $u_k \subset  W^{1,p}(Q_\rho^\nu(x))$, $0\le u_k\le1$, such that $u_k=\bar u_{x,\e_k}^\nu$ in a neighbourhood of $\partial Q_\rho^\nu(x)$, and
%\begin{equation}\label{uniformbound:me}
%\m_k(\bar u_{x,\e_k}^\nu,Q_\rho^\nu(x))\leq c_2C_\uu\rho^{n-1}.
%\end{equation}
%
Up to estracting a subsequence we may assume that
\[
\lim_{k\to +\infty} \F_{{k}}(u_{k},Q_\rho^\nu(x)) = \liminf_{k \to +\infty}\F_{k}(u_k,Q_\rho^\nu(x))\,.
\]
Then extend $u_k$ to a $W^{1,p}_{\loc}(\R^n)$-function by setting
\begin{align*}
w_k\defas
\begin{cases}
u_k &\text{in } Q_\rho^\nu(x)\,,\\
\bar u_{x,\e_k}^\nu &\text{in } \R^n\sm Q_\rho^\nu(x)\,.
\end{cases}
\end{align*}

From ~\eqref{uniformbound:me}, \eqref{est:bounds-Fe} and \cite[Lemma 4.1]{MoMo} there exists a subsequence (not relabelled) such that
\[w_{k}\to u\quad  \text{in}\quad   L^p_{\loc}(\R^n)\,,\]
for some $u\in L^p_{\loc}(\R^n;\R^m)\cap BV(Q^\nu_{(1+\eta)\rho}(x);\{0,1\})$ such that $u=u_{x}^\nu$ in a neighbourhood of $\partial Q^\nu_{(1+\eta)\rho}(x)$.
Hence by $\Gamma$-convergence, \eqref{1dim-energy-bis} and \eqref{eq:q-min-e} we obtain
\begin{align*}
\F_\infty(u,Q^\nu_{(1+\eta)\rho}(x)) &\leq \liminf_{k\to +\infty} \F_{{k_j}}(w_{k},Q_{(1+\eta)\rho}^\nu(x))\\
&\leq\lim_{k \to +\infty} \F_{k_j}(u_{k},Q_{\rho}^\nu(x))+c_2C_\uu\big((1+\eta)^{n-1}-1)\rho^{n-1}\\
&\leq \liminf_{k\to+\infty}\m_k(\bar u_{x,\e_k}^\nu,Q_\rho^\nu(x))+\eta\rho^{n-1}+c_2C_\uu\big((1+\eta)^{n-1}-1\big)\rho^{n-1}\,.
\end{align*}
Combining the above inequality with
\begin{equation*}\label{eq:trivial-ineq}
\m(u^\nu_{x},Q^\nu_{(1+\eta)\rho}(x)) \leq \F_\infty(u,Q^\nu_{(1+\eta)\rho}(x))\,, 
\end{equation*}
dividing by $\rho^{n-1}$ and passing to the limsup as $\rho\to0$ we infer
\begin{equation*}
(1+\eta)^{n-1}{\hat f}(x,\nu)\leq
f'(x,\nu)
+\eta+c_2C_\uu\big((1+\eta)^{n-1}-1\big)\,.
\end{equation*}
Eventually we conclude by sending $\eta\to0$.

\medskip

\emph{Step 2:} In this step we show that $\hat f(x,\nu)\ge f''(x,\nu)$ for every $x\in\R^n$, $\nu\in\Sph^{n-1}$, and $\rho>0$.\\
We fix $\eta>0$ and choose $u\in BV(Q_\rho^\nu(x);\{0,1\})$ with $u=u_{x}^\nu$ near $\partial Q_\rho^\nu(x)$ and
\begin{equation}\label{eq:q-min-m}
\F_\infty(u,Q_\rho^\nu(x))\leq\m(u_{x}^\nu,Q_\rho^\nu(x))+\eta\,.
\end{equation}
We extend $u$ to the whole $\R^n$ to the jump function $u_{x}^\nu$ in $\R^n\sm Q_\rho^\nu(x)$. 
By $\Gamma$-convergence there exists a sequence $u_k$ converging to $u$ in $L^1_\loc(\R^n)$ such that
\begin{equation}\label{approx:min-rec}
\lim_{k \to +\infty}\F_{k}(u_k,Q_\rho^\nu(x))= \F_\infty(u,Q_\rho^\nu(x))\,.
\end{equation}
We next properly modify the sequence $u_k$ in order to obtain a new sequence $\hat u_k\in\Adm(\bar{u}_{x,\e_k}^\nu,A)$. 
%This is done by appealing to the fundamental estimate Proposition~\ref{prop:fund-est}.
To this end let $0<\rho''<\rho'<\rho$ be such that $u=u_{x}^\nu$ on $Q_\rho^\nu(x)\sm\overline{Q}_{\rho''}^\nu(x)$. Then we apply Proposition~\ref{prop:fund-est} with $A=Q_{\rho''}^\nu(x)$, $A'=Q_{\rho'}^\nu(x)$, $B=Q_\rho^\nu(x)\sm\overline{Q}_{\rho''}^\nu(x)$ and $u=u_k$, $v=\bar u^\nu_{x,\e_k}$. In this way to get a new sequence $\hat{u}_k\subset W^{1,p}_\loc(\R^n)$ converging to $u$ in $L^p(Q_\rho^\nu(x))$ such that 
$$\hat{u}_k=u_k\quad\text{in} \quad Q_{\rho''}^\nu(x),\quad \hat{u}_k=\bar u_{x,\e_k}^\nu\quad\text{in} \quad Q_\rho^\nu(x)\sm\overline{Q}_{\rho'}^\nu(x),$$
and
\begin{equation}\label{approx:min-bc}
\begin{split}
\limsup_{k\to+\infty}\F_{k}(\hat{u}_k,Q_\rho^\nu(x))&\leq \limsup_{k\to+\infty}\Big(\F_{k}(u_k,Q_{\rho'}^\nu(x))+\F_{k}(\bar u_{x,\e_k}^\nu, Q_{\rho}^\nu(x)\sm\overline{Q}_{\rho''}^\nu(x))\Big)\\
&\leq \F_\infty(u,Q_\rho^\nu(x)) +c_2C_\uu\L^{n-1}\big(Q_{\rho}'\sm\overline{Q'}_{\!\!\rho''}\big)\,.
\end{split}
\end{equation}
Here the second inequality follows  from~\eqref{approx:min-rec} and~\eqref{1dim-energy-bis} together with~\eqref{f-value-at-0}.
 
Recalling that $\hat u_k$ is a test function for the definition of $\m_k(\bar u_{x,\e_k}^\nu,Q_\rho^\nu(x))$, combining~\eqref{eq:q-min-m} and~\eqref{approx:min-bc} we get
\begin{align*}
\limsup_{k\to+\infty}\m_k(\bar u_{x,\e_k}^\nu,Q_\rho^\nu(x))\EEE &\leq \m(u_{x}^\nu,Q_\rho^\nu(x))+\eta+C  (\rho^{n-1}-(\rho'')^{n-1})\,.
\end{align*}
We divide the above inequality by $\rho^{n-1}$ and conclude by passing t to the limit first as $\rho''\to\rho$, second as $\eta\to0$ and eventually to the limsup as $\rho \to 0$. 
% Then letting $\rho''\to\rho$, by the arbitrariness of $\eta>0$ we get
% \begin{equation}\label{ub:m-me-1}
% \limsup_{k\to+\infty}\m_k(\bar u_{x,\zeta,\e_k}^\nu,Q_\rho^\nu(x))\leq\m(u_{x,\zeta}^\nu,Q_\rho^\nu(x))\,.
% \end{equation}
% %
% Eventually by dividing both sides of \eqref{ub:m-me-1} by $\rho^{n-1}$ and passing to the limsup as $\rho \to 0$ we conclude. 
\end{proof}
%

%%%%%%%%%%%%%%%%%%%%%%%%%%%%%%%%%%%%%%%%%
%
%Properties satisfied by f',f''
%
%%%%%%%%%%%%%%%%%%%%%%%%%%%%%%%%%%%%%%%%%%%
\section{Properties of $f',f''$}\label{sect:prop}
\noindent This section is devoted to prove some properties satisfied by the functions $f'$, and $f''$ defined in \eqref{f'} and  \eqref{f''}  respectively.  To this purpose is convenient to characterise them in an alternative way.
%\begin{remark} \label{rem:prop-f_infty}
%By Proposition \ref{p:surface-term} and Proposition \ref{prop:f'-f''} property \eqref{prop-f_infty} readily follows.
%\end{remark}
Let us fix first some notation.
For $\rho,\delta,\e_k>0$ with $\rho> \delta>2\e_k$ we set
\begin{equation*}%\label{eq:ms-delta}
	\m_{k}^{\delta}(\bar u_{x,\e_k}^\nu,Q_\rho^\nu(x))\defas\inf\{\F_k(u,Q_\rho^\nu(x))\colon u\in \Adm^\delta(\bar u_{x,\e_k}^\nu,Q_\rho^\nu(x))\},
\end{equation*}
where 
\begin{equation*}%\label{c:adm-e-d}
	\Adm^\delta(\bar u_{x,\e_k}^\nu,Q_\rho^\nu(x))\defas\{u\in W^{1,p}(Q_\rho^\nu(x)),\; 0\leq u\leq 1 \colon   u=\bar u_{x,\e_k}^\nu\ \text{in}\ Q_\rho^\nu(x) \setminus \overline Q_{\rho-\delta}^\nu(x)\}\,.
\end{equation*}
Moreover, let $f'_\rho,f''_\rho:\R^n\x\S^{n-1}\to[0,+\infty]$ be the functions defined as
\begin{equation}\label{c:f'-rho}
	\begin{split}
		f'_\rho(x,\nu) &\defas \inf_{\delta>0}\liminf_{k\to+\infty}\m_{k}^{\delta}(\bar u_{x,\e_k}^\nu,Q_\rho^\nu(x))= \lim_{\delta\to 0}\liminf_{k\to+\infty}\m_{k}^{\delta}(\bar u_{x,\e_k}^\nu,Q_\rho^\nu(x))\,,\\[4pt]
		f''_\rho(x,\nu) &\defas \inf_{\delta>0}\limsup_{k\to+\infty}\m_{k}^{\delta}(\bar u_{x,\e_k}^\nu,Q_\rho^\nu(x))=\lim_{\delta\to 0}\limsup_{k\to+\infty}\m_{k}^{\delta}(\bar u_{x,\e_k}^\nu,Q_\rho^\nu(x))\,.
	\end{split}
\end{equation}
%In addition thanks to Proposition~\ref{prop:fund-est} it can be shown that  
%\begin{equation*}
%	f'(x,\nu)=\limsup_{\rho\to 0}\frac{1}{\rho^{n-1}}\liminf_{k\to +\infty}\m_{k,N}( u_{x}^\nu,Q^\nu_\rho(x))\, ,
%\end{equation*}
%%
%\begin{equation*}
%	f''(x,\nu)=\limsup_{\rho\to 0}\frac{1}{\rho^{n-1}}\limsup_{k\to +\infty}\m_{k,N}( u_{x}^\nu,Q^\nu_\rho(x))\,,
%\end{equation*}
%where 
%\begin{equation*}
%	\m_{k,N}( u_{x}^\nu,Q^\nu_\rho(x))\defas\inf\{\F_k(u,Q^\nu_\rho(x))\colon u\in \Adm_{\e_k,\rho}({u}_{x}^\nu,\nu)\}\,,
%\end{equation*}
%and 
%\begin{equation*}
%	\Adm_{\e_k,\rho}({u}_{x}^\nu,\nu)\defas\{u\in W^{1,p}(Q^\nu_\rho(x)),\ 0\le u\le 1\colon u=u^\nu_x\text{ near } \partial Q^\nu_\rho(x) \text{ in } \{|(y-x)\cdot\nu|>\e_k\}\}\,.
%\end{equation*}
%A similar result was proven in \cite[Proposition 2.6]{BMZ} in the context of  elliptic functionals of  Ambrosio-Tortorelli type by resorting to a more complex argument.
Next we prove the following technical Lemma.
\begin{lemma}\label{lemma:f'-f''}
	Let $f'_\rho,f''_\rho$ be as in \eqref{c:f'-rho}. Then the following hold:
	\begin{enumerate}[label=$(\roman*)$]
		\item \label{i}The restrictions of $f'_\rho$, $f''_\rho$ to the sets $\R^n\x\widehat{\S}^{n-1}_+$ and $\R^n\x\widehat{\S}^{n-1}_-$ are upper semicontinuous;
		\item\label{ii} The functions $\rho \to (f'_{\rho}(x,\nu)-c_2C_\uu\rho^{n-1})$ and $\rho \to (f''_{\rho}(x,\nu)-c_2C_\uu\rho^{n-1})$ are nonincreasing on $(0,+\infty)$;
		\item\label{iii} For every $x\in \R^n$ and every $\nu \in \S^{n-1}$ there hold
		\begin{equation*}
			f'(x,\nu)=\limsup_{\rho\to 0}\frac{1}{\rho^{n-1}}f_\rho'(x,\nu)\;\text{ and }\; f''(x,\nu)=\limsup_{\rho\to 0}\frac{1}{\rho^{n-1}}f_\rho''(x,\nu)\,,
		\end{equation*}
		with $f'$ and $f''$ as in \eqref{f'} and \eqref{f''} respectively.
	\end{enumerate}
\end{lemma}
\begin{proof} We prove the statement only for $f'_\rho$ and we show \ref{i} only for its restriction to the set $\R^n\x\widehat{\S}^{n-1}_+$, the other cases can be treated arguing similarly. \\
	
	\step 1 we show the validity of \ref{i}.
	Let $\rho>0$, $x\in\R^n$, and $\nu\in\widehat{\S}^{n-1}_+$ be fixed. 	Let $(x_j, \nu_j)\subset\R^n \times \widehat{\S}^{n-1}_+$ be  a sequence converging to $(x,\nu)$. We want to show that 
			\begin{align}\label{eq:upp-semi-cont}
			\limsup_{j\to+\infty}f_\rho'(x_j,\nu_j)\leq f_\rho'(x,\nu)\,.
		\end{align}
To this aim we notice that \eqref{c:f'-rho} implies the following:
for fixed $\eta>0$ there exists $\delta_\eta>0$ such that  
	\begin{equation}\label{c:delta0-almost-optimal}
		\liminf_{k \to +\infty }\m_k^{\delta}(\bar{u}_{x,\e_k}^\nu,Q_\rho^\nu(x))\leq f_\rho'(x,\nu)+\eta,
	\end{equation}
	for every $\delta\in(0,\delta_\eta)$.  Hence we can choose $\delta_0$ such that  $3\delta_0\in(0,{\delta_\eta})$ and  \eqref{c:delta0-almost-optimal} holds true with $\delta=3\delta_0$.
% 	\begin{equation}\label{delta0-almost-optimal}
% 		\liminf_{k\to+\infty}\m_{k}^{3\delta_0}(\bar u_{x,\e_k}^\nu,Q_\rho^\nu(x))\leq f_\rho'(x,\nu)+\eta\,.
% 	\end{equation}
	%
	Let also $u_k \in \Adm^{3\delta_0}(\bar{u}_{x,\e_k}^\nu,Q_\rho^\nu(x))$ that satisfies
	\begin{equation}\label{almostmin:uppersemicont}
		\F_{k}(u_k,Q_\rho^\nu(x))\leq\m_{k}^{3\delta_0}(u_x^\nu,Q_\rho^\nu(x))+\eta\,.
	\end{equation}
	In what follows we modify $u_k$ in order to obtain a new sequence $\tilde{u}_k \in \Adm^\delta(\bar{u}_{x_j,\e_k}^{\nu_j},Q_\rho^{\nu_j}(x_j))$ for $j$ large enough and any $\delta\in (0,\delta_0)$ without essentially increasing the energy.
We define
\[
h_{k,j}\defas\e_k+|x-x_j|+\frac{\sqrt{n}}{2}\rho|\nu-\nu_j|
\]
and
\begin{equation*}
	R_{k,j}\defas R_\nu\Big(\big(Q'_{\rho-2\delta_0}\sm\overline{Q}'_{\rho-3\delta_0}\big)\x\big(-h_{k,j},h_{k,j}\big)\Big)+x\,.
\end{equation*} 
Since the map $\nu\mapsto R_\nu$ is continuous, there exists $\hat\jmath= \hat\jmath(\delta_0)\in\N$ such that 
	\begin{equation}\label{inclusion:uppersemicont}
	Q_{\rho-5\delta_0}^{\nu_j}(x_j)\subset Q_{\rho-4\delta_0}^\nu(x)\subset	Q_{\rho-2\delta_0}^\nu(x)\subset Q_{\rho-\delta_0}^{\nu_j}(x_j) \,,
	\end{equation}
and
%Up to increasing $\hat \jmath$, we can assume that 
$$R_{k,j}\subset Q_{\rho-2\delta_0}^\nu(x)\sm\overline{Q}^\nu_{\rho-3\delta_0}(x),$$ for every $j\geq \hat \jmath$. 
For later convenience we observe that 
%\begin{equation}\label{meas-R}
%	 \L^n(R_{k,j})\le C\e_k(\rho-3\delta_0)^{n-2}h_{k,j}\,,
%\end{equation}
%and
\begin{equation}\label{ouside-R}
u_k(y)=\bar u^\nu_{x,\e_k}(y)= \bar u^{\nu_j}_{x_j,\e_k}(y) \in\{0,1\}\;\text{ for a.e.}\; y\in\Big(Q_{\rho-2\delta_0}^\nu(x)\sm \overline{Q}_{\rho-3\delta_0}^\nu(x)\Big)\sm \overline{R}_{k,j}\,.
\end{equation}
Indeed we have
\begin{equation}\label{indeed}
  \bar u^\nu_{x,\e_k}=u^\nu_x=u^{\nu_j}_{x_j}=\bar u^{\nu_j}_{x_j,\e_k}\; \text{ in }\;\{(y-x)\cdot\nu\ge\e_k\}\cap\{(y-x_j)\cdot\nu_j\ge\e_k\}\,.  
\end{equation}
On the other hand for $y\in \Big(Q_{\rho-3\delta_0+2\e_k}^\nu(x)\sm \overline{Q}_{\rho-3\delta_0}^\nu(x)\Big)\sm \overline{R}_{k,j}$ it holds
\begin{equation*}\label{scalarproduct1}
	|(y-x)\cdot\nu|=|R_\nu^T(y-x)\cdot e_n|>h_{k,j}>\e_k\,,
\end{equation*}
and  by applying the triangular inequality twice we have
\begin{equation*}\label{scalarproduct2}
	\begin{split}
		|(y-x_j)\cdot\nu_j| &> 
		h_{k,j}-|x-x_j|-|y-x||\nu-\nu_j|
		\\ 
		&\ge \e_k+(\frac{\sqrt{n}}{2}\rho-|y-x|)|\nu-\nu_j|
		\geq\e_k\,,
	\end{split}
\end{equation*}
where the last inequality is a consequence of $Q_\rho^\nu(x)\subset B_{\frac{\sqrt{n}}{2}\rho}(x)$.\\
%
%\BBB
%Now let $\tilde u_k\in W^{1,p}(Q_\rho^{\nu_j}(x_j))$, $0\le \tilde u_k\le1$, be given by
%\begin{equation*}
%\tilde u_k\defas u_k\varphi+ \bar u^{\nu_j}_{x_j,\e_k}(1-\varphi)\,,
%\end{equation*}
%where $\varphi$ is a smooth cut-off between $ Q^\nu_{\rho-3\delta_0}(x)$ and $Q^\nu_{\rho-2\delta_0}(x)$. Thus using \eqref{inclusion:uppersemicont} we have
%\begin{equation}\label{estimate-energy}
%	\begin{split}
%		\F_k(\tilde u_k,Q_\rho^{\nu_j}(x_j))\le \F_k(u_k,Q^\nu_\rho(x))+ \F_k(\bar u^{\nu_j}_{x_j,\e_k}&,Q_\rho^{\nu_j}(x_j)\setminus\overline Q^{\nu_j}_{\rho-5\delta_0}(x_j))\\
%		&
%		+	\F_k(\tilde u_k,Q^\nu_{\rho-2\delta_0}(x)\setminus \overline Q^\nu_{\rho-3\delta_0}(x))
%		\,.  
%	\end{split}
%\end{equation}
%Moreover from  \eqref{ouside-R} and \eqref{est:bounds-Fe} the last term on the right-hand side can be estimated as follows
%\begin{equation*}
%	\begin{split}
%		\F_k(\tilde u_k,Q^\nu_{\rho-2\delta_0}(x)\setminus \overline Q^\nu_{\rho-3\delta_0}(x))=	\F_k(\tilde u_k,R_{k,j})\le c_2 \M_k(\tilde u_k,R_{k,j})\le 
%	\end{split}
%\end{equation*}
%\EEE
Next we apply Proposition \ref{prop:fund-est} with 
\begin{equation*}
A\defas Q^\nu_{\rho-3\delta_0}(x)\,,\quad
A'\defas Q^\nu_{\rho-2\delta_0}(x)\,,\quad
B\defas Q_\rho^{\nu_j}(x_j)\setminus\overline Q^{\nu_j}_{\rho-5\delta_0}(x_j)\,,
\end{equation*}
and $u\defas u_k$, $v\defas \bar u^{\nu_j}_{x_j,\e_k}$. 
Note that thanks to \eqref{inclusion:uppersemicont} the following hold
\begin{equation*}\label{notethat}
    A\cup B=Q_\rho^{\nu_j}(x_j)\,,\quad
    B\setminus\overline A'=Q_\rho^{\nu_j}(x_j)\setminus\overline Q^\nu_{\rho-2\delta_0}(x)\,,
\end{equation*}
and
\begin{equation}\label{notethatbis}
    (A'\setminus\overline A)\cap B=A'\setminus\overline A=Q^\nu_{\rho-2\delta_0}(x)\setminus \overline Q^\nu_{\rho-3\delta_0}(x)\,.
\end{equation}
\medskip 

\noindent
In particular Proposition \ref{prop:fund-est} \ref{fund-i} provides us with $\tilde u_k\in W^{1,p}(Q_\rho^{\nu_j}(x_j))$, $0\le \tilde u_k\le1$ with
\begin{equation}\label{bdry}
\tilde u_k=u_k\quad\text{ a.e. in } Q^\nu_{\rho-3\delta_0}(x)\,,
\quad
\tilde u_k=\bar u^{\nu_j}_{x_j,\e_k}\quad\text{ a.e. in }
Q_\rho^{\nu_j}(x_j)\setminus\overline Q^\nu_{\rho-2\delta_0}(x)\,,
\end{equation}
 and such that
\begin{equation}\label{estimate-energy}
\begin{split}
  \F_k(\tilde u_k,Q_\rho^{\nu_j}(x_j))\le(1+K\e_k) \Big(\F_k(u_k,Q^\nu_\rho(x))+ \F_k(\bar u^{\nu_j}_{x_j,\e_k}&,Q_\rho^{\nu_j}(x_j)\setminus\overline Q^{\nu_j}_{\rho-5\delta_0}(x_j))\Big)\\
  &
+ \omega_k(u_k,u^{\nu_j}_{x_j,\e_k},A,A',B)\,.  
\end{split}
\end{equation}
Moreover from \eqref{notethatbis}, property \eqref{ouside-R} becomes
$$u_k(y)=u^{\nu_j}_{x_j,\e_k}(y)\in\{0,1\} \quad\text{ for a.e. $y$ in }  ((A'\setminus\overline A)\cap B)\setminus R_{k,j}\,.$$
Thus, applying Proposition \ref{prop:fund-est} \ref{fund-ii} with $\tilde S=((A'\setminus\overline A)\cap B)\setminus R_{k,j}$
yields
\begin{equation}\label{estimate-resto}
\omega_k(u_k,u^{\nu_j}_{x_j,\e_k},A,A',B)
\le K\L^n(R_{k,j})\,.
\end{equation}
Now \eqref{1dim-energy-bis} together with \eqref{f-value-at-0} give
\begin{equation}\label{primo}
\F_k(\bar u^{\nu_j}_{x_j,\e_k},Q_\rho^{\nu_j}(x_j)\setminus\overline Q^{\nu_j}_{\rho-5\delta_0}(x_j))
\le c_2C_{\uu}\L^{n-1}\Big(Q'_{\rho}\sm\overline{Q'}_{\!\!\rho-5\delta_0}\big)\Big)\leq C\delta_0\rho^{n-2}\,,
\end{equation}
while
\begin{equation}\label{secondo}
	\L^n(R_{k,j})\le C\delta_0\rho^{n-2}h_{k,j}\,.
	%\to0\quad\text{ as }\ k\to\infty\,.
\end{equation}
By \eqref{bdry} we have $\tilde{u}_k \in \Adm^\delta(\bar{u}_{x_j,\e_k}^{\nu_j},Q_\rho^{\nu_j}(x_j))$ for any $j\geq \hat \jmath$ and any $\delta\in (0,\delta_0)$.
Hence gathering \eqref{almostmin:uppersemicont},  \eqref{estimate-energy}, \eqref{estimate-resto},\eqref{primo}, and \eqref{secondo} we get
\begin{equation}\label{est:uppersemicont-final}
\begin{split}
\m_{k}^{\delta}(\bar u_{x_j,\e_k}^{\nu_j},Q_\rho^{\nu_j}(x_j))\leq (1+C\e_k)(\m_{k}^{3\delta_0}(\bar u_{x,\e_k}^\nu,Q_\rho^\nu(x))+\eta)
+C\rho^{n-2}\delta_0(1+h_{k,j})\,.
\end{split}
\end{equation}
Using~\eqref{c:delta0-almost-optimal} with $\delta=3\delta_0$ and passing to the limit  in~\eqref{est:uppersemicont-final} first  as $k \to +\infty$, then  as $\delta \to 0$, and finally as $j \to +\infty$ gives
\begin{align*}
	\limsup_{j\to+\infty}f_\rho'(x_j,\nu_j)\leq f_\rho'(x,\nu)+2\eta+C\delta_0\rho^{n-2}\,,
\end{align*}
since $\lim_j\lim_kh_{k,j}=0$.
By letting $\delta_0\to 0$ \eqref{eq:upp-semi-cont} and then $\eta\to0$ we finally conclude.
\medskip

\step 2  we show \ref{ii}.
 Let $0<\rho\le\rho'$, $x\in \R^n$, $\nu\in S^{n-1}$ be fixed. We show that 
 \begin{equation}\label{nonincreasing}
 	f'_{\rho'}(x,\nu)-c_2C_\uu(\rho')^{n-1}
 	\le f'_{\rho}(x,\nu)- c_2C_\uu(\rho)^{n-1}\,.
 \end{equation}
 Let $\eta>0$ be fixed and let $\delta_\eta>0$ be such that \eqref{c:delta0-almost-optimal} holds true for every $\delta\in(0,\delta_\eta)$. For fixed $\delta\in(0,\delta_\eta)$ let also $u_k \in \Adm^{\delta}(\bar{u}_{x,\e_k}^\nu,Q_{\rho'}^\nu(x))$ that satisfies
\begin{equation}\label{almostmin:uppersemicont_bis}
	\F_{k}(u_k,Q_{\rho}^\nu(x))\leq\m_{k}^{\delta}(u_x^\nu,Q_{\rho}^\nu(x))+\eta\,.
\end{equation}
we now extend $u_k$, without relabelling it, to $\bar u^\nu_{x,\e_k}$ in $Q^\nu_{\rho'}(x)\setminus \overline Q^\nu_{\rho}(x)$ so that it belongs to the class $\Adm^{\delta}(\bar{u}_{x,\e_k}^\nu,Q_{\rho'}^\nu(x))$.
 Then \eqref{f-value-at-0} and \eqref{1dim-energy-bis} yield
\begin{equation}\label{est:extension_bis}
	\begin{split}
		\F_k(u_k,Q_{\rho'}^\nu(x)) &\leq\F_k(u_k,Q_{\rho}^\nu(x))+\F_k\big(\bar{u}_{x,\e_k}^\nu,(Q_{\rho'}^\nu(x)\sm\overline{Q}_{\rho}^\nu(x)\big)\\
		&\leq \F_k(u_k,Q_{\rho}^\nu(x))+c_2C_\uu((\rho')^{n-1}-(\rho)^{n-1})\,.
	\end{split}
\end{equation}
Finally letting first $k\to\infty$ and then $\delta\to0$ from \eqref{c:delta0-almost-optimal} and \eqref{almostmin:uppersemicont_bis} we get
\begin{equation*}
f'_{\rho'}(x,\nu)\le f'_{\rho}(x,\nu)+ c_2C_\uu((\rho')^{n-1}-(\rho)^{n-1})+2\eta\,.
\end{equation*}
Eventually by sending $\eta\to0$ we deduce \eqref{nonincreasing}.
\medskip

\step 3  we show \ref{iii}.
	One inequality follows immediately. Indeed as $\Adm^\delta(\bar{u}_{x,\e_k}^\nu,Q_\rho^\nu(x))\subset\Adm(\bar{u}_{x,\e_k}^\nu,Q_\rho^\nu(x))$ for every $\delta \in (0,\rho)$, we have
\[
f'(x,\nu)\leq\limsup_{\rho\to 0}\frac{1}{\rho^{n-1}}f_\rho'(x,\nu)\,.
\] 
It remains to prove the opposite inequality.
For fixed $\rho> 0$, $x\in\R^n$, and $\nu\in\S^{n-1}$ and for every $k\in \N$ such that $\e_k\in(0,\frac{\rho}{2})$ let $u_k\in\Adm(\bar{u}_{x,\e_k}^\nu,Q_\rho^\nu(x))$ satisfy 
\begin{equation}\label{cond:v-extension}
	\F_k(u_k,Q_\rho^\nu(x))\leq\m_k(\bar{u}_{x,\e_k}^\nu,Q_\rho^\nu(x))+\rho^n\,,
\end{equation}
Fix $\alpha>0$, let $\rho_\alpha\defas(1+\alpha)\rho$, and extend $u_k$, without relabelling it, to $\bar{u}_{x,\e_k}^\nu$ in  $Q_{\rho_\alpha}^\nu(x)\sm\overline{Q}_\rho^\nu(x)$. In this way $u_k\in\Adm^\delta(\bar{u}_{x,\e_k}^\nu,Q_{\rho_\alpha}^\nu(x))$  for any $\delta\in(0,\alpha\rho)$ and similarly to \eqref{est:extension_bis} it holds
\begin{equation}\label{est:extension}
	\begin{split}
		\F_k(u_k,Q_{\rho_\alpha}^\nu(x)) &\leq\F_k(u_k,Q_\rho^\nu(x))+\F_k\big(\bar{u}_{x,\e_k}^\nu,(Q_{\rho_\alpha}^\nu(x)\sm\overline{Q}_\rho^\nu(x)\big)\\
		&\leq \F_k(u_k,Q_\rho^\nu(x))+c_2C_\uu((1+\alpha)^{n-1}-1))\rho^{n-1}\,.
	\end{split}
\end{equation}
 Now~\eqref{cond:v-extension} and~\eqref{est:extension} give
\begin{equation*}
	\inf_{\delta>0}\liminf_{k\to+\infty}\m_k^{\delta}(\bar{u}_{x,\e_k}^\nu,Q_{(1+\alpha)\rho}^\nu(x))\leq\liminf_{k\to+\infty}\m_k(\bar{u}_{x,\e_k}^\nu,Q_\rho^\nu(x))+\rho^n+c_2C_\uu((1+\alpha)^{n-1}-1))\rho^{n-1}\,.
\end{equation*}
Rescaling the above inequality by $((1+\alpha)\rho)^{n-1}$ and passing to the limsup as $\rho\to 0$ we infer
\[
(1+\alpha)^{n-1}\limsup_{\rho \to 0}\frac{1}{\rho^{n-1}}f_\rho'(x,\nu)\leq f'(x,\nu)+c_2C_\uu((1+\alpha)^{n-1}-1))\,.
\] 
We finally conclude by the arbitrariness of $\alpha>0$. 
\end{proof}
%\PPP
%	Let $\rho>0$, $x\in\R^n$, and $\nu\in\widehat{\S}^{n-1}_+$ be fixed. Let $\delta'>\delta$, since $\Adm^{\delta'}(\bar{u}_{x,\e_k}^\nu,Q_\rho^\nu(x))\subset\Adm^\delta(\bar{u}_{x,\e_k}^\nu,Q_\rho^\nu(x))$, we immediately obtain that in~\eqref{c:f'-rho} the infimum in $\delta>0$ coincides with the limit as $\delta\to 0$, which in particular exists. Thus, for fixed $\eta>0$ there exists $\delta_\eta>0$ such that  
%\begin{equation}
%	\liminf_{k \to +\infty }\m_k^{\delta}(\bar{u}_{x,\e_k}^\nu,Q_\rho^\nu(x))\leq g_\rho'(x,\nu)+\eta,
%\end{equation}
%for every $\delta\in(0,\delta_\eta)$. 
%\EEE
%We are now ready to state and prove the following proposition which establishes the properties satisfied by $f'$ and $f''$.
%\EEE
We now state the main result of this section: 
\begin{proposition}\label{prop:f'-f''}
	Let $(f_k) \subset \mathcal F$; then the functions $f'$ and $f''$ defined, respectively, as in \eqref{f'} and \eqref{f''} are Borel measurable and satisfy the following: for every $x\in \R^n$ and every $\nu \in \Sph^{n-1}$ it holds
		\begin{equation}\label{c:bdds-f'-f''}
			c_1 c_p \leq f'(x, \nu) \leq c_2 c_p, \quad c_1 c_p \leq f''(x,\nu) \leq c_2 c_p, 
		\end{equation}
		with $$c_p:=p(p-1)^{\frac{1-p}{p}}\int_0^1W(t)^{\frac{p-1}{p}}\dt\,.$$
% 	\begin{enumerate}[label=$(\arabic*)$]
% 		\item (symmetry) for every $x\in \R^n$ and every $\nu \in \S^{n-1}$ it holds
% 		\begin{equation}\label{pr:symmetry}
% 			f'(x, \nu)=f'(x,-\nu), \quad f''(x, \nu)=f''(x,-\nu);
% 		\end{equation}
% 		\item (boundedness) for every $x\in \R^n$ and every $\nu \in \Sph^{n-1}$ it holds
% 		\begin{equation}\label{c:bdds-f'-f''}
% 			c_1 c_p \leq f'(x, \nu) \leq c_2 c_p, \quad c_1 c_p \leq f''(x,\nu) \leq c_2 c_p, 
% 		\end{equation}
% 		where $$c_p:=p(p-1)^{\frac{1-p}{p}}\int_0^1W(t)^{\frac{p-1}{p}}\dt\,.$$
%	\end{enumerate}
\end{proposition}

\begin{proof}
%[Proof of  Proposition \ref{prop:f'-f''}]
We prove the statement only for $f'$, 
as the proof for $f''$ can be achieved in a similar way. 
	
		\step 1 we show that $f'$ is Borel measurable. 
	Let $\rho>0$ and let $f'_\rho$ be the function defined in \eqref{c:f'-rho}.  
	By Lemma \ref{lemma:f'-f''} \ref{ii} we have that the function $\rho \to f'_{\rho}(x,\nu)-c_2C_\uu\rho^{n-1}$ is nonincreasing on $(0,+\infty)$, hence in particular 
	\[
	\lim_{\rho'\to \rho^-}f'_{\rho'}(x,\nu)\geq f'_{\rho}(x,\nu)\geq \lim_{\rho'\to \rho^+} f'_{\rho'}(x,\nu),
	\]
	for every $x\in \R^n$, $\nu \in \Sph^{n-1}$, and every $\rho>0$. This together with Lemma \ref{lemma:f'-f''} \ref{iii} imply that
	\[f'(x,\nu)=
	\limsup_{\rho\to 0}\frac{1}{\rho^{n-1}}f'_{\rho}(x,\nu)= \limsup_{\substack{\rho\to 0\\\rho \in D}}\frac{1}{\rho^{n-1}}f'_{\rho}(x,\nu)
	\]
where $D$ is any countable dense subset of $(0,+\infty)$.  Since by Lemma \ref{lemma:f'-f''} \ref{i} the function
	$(x,\nu) \mapsto f'_\rho(x,\nu)$ is Borel measurable for every $\rho>0$, and the limit on a countable set of Borel functions is Borel we conclude.
	% 	and hence by  we get
% 	\begin{align*}
% 		f'(x,\nu)&=\limsup_{\substack{\rho\to 0\\ \rho \in D}}\frac{1}{\rho^{n-1}} f'_{\rho}(x,\nu).
% 	\end{align*}
% 	Therefore the Borel measurability of $f'$ follows by Lemma \ref{lemma:f'-f''} \ref{i} which guarantees, in particular, that the function
% 	$(x,\nu) \mapsto f'_\rho(x,\nu)$ is Borel measurable for every $\rho>0$. 
	
% 	\step 2 we show that $f'$ is symmetric in $\nu$.  Property \eqref{pr:symmetry} immediately follows from the definition of $f'$ and from the fact that $\bar u^\nu_{x,\e_k}=-\bar u^{-\nu}_{x,\e_k}+1$ a.e.\ and $Q^\nu_\rho(x)=Q^{-\nu}_\rho(x)$ (see \ref{Rn}), which implies in particular that $u\in\Adm(\bar u_{x,\e_k}^\nu,Q^\nu_\rho(x))$  if and only if  $(-u+1)\in\Adm(\bar u_{x,\e_k}^{-\nu},Q^{-\nu}_\rho(x))$. 

	\step 2 we show that  $f'$ is bounded. For every $x\in\R^n$, $\nu\in\S ^{n-1}$, and $\rho>0$ we define the minimisation problem
	%\begin{equation*}%\label{c:min-bd}
	%\m_{k,\rho}(x,\nu) \defas \min\bigg\{\int_{Q^\nu_\rho(x)} \bigg(\frac{(1-v)^p}{\e_k}+\e_k^{p-1}|\nabla v|^p\bigg)\dy \colon v\in \Adm_{\e_k,\rho}(x,\nu)\bigg\}.
	%\end{equation*}
	\begin{equation*}%\label{c:min-bd}
		\m_{k,\rho}(x,\nu) \defas \min\bigg\{\mathcal M_k(u,Q^\nu_\rho(x))
		\colon u\in \Adm(\bar u_{x,\e_k}^{\nu},Q^{\nu}_\rho(x))
		\bigg\},
	\end{equation*}
	where $\mathcal M_k$ is defined as in \eqref{MM}.
From Remark \ref{initial-rem} \ref{rem:bounds-Fe} we have 
	\begin{equation*}
	    c_1 	\m_{k,\rho}(x,\nu)\le \m_{k}(\bar u_{x,\e_k}^\nu,Q^\nu_\rho(x))\le 
	    c_2	\m_{k,\rho}(x,\nu).
	\end{equation*}
	Therefore to prove that $f'$ satisfies the bounds in \eqref{c:bdds-f'-f''} it is enough to show that
	\begin{equation}\label{c:claim}
		\lim_{k \to +\infty} \m_{k,\rho}(x,\nu)=c_p\rho^{n-1}.
	\end{equation}
	% 	we start by observing that thanks to Remark \ref{initial-rem} \ref{rem:bounds-Fe} we have 
% 	%\[
% 	%c_1 \int_{Q^\nu_\rho(x)} \bigg(\frac{W(u)}{\e_k}+\e_k^{p-1}|\nabla u|^p\bigg)\dx \leq \F_{k}(u,Q^\nu_\rho(x)) \leq c_2 \int_{Q^\nu_\rho(x)} \bigg(\frac{W(u)}{\e_k}+\e_k^{p-1}|\nabla u|^p\bigg)\dx,
% 	%\]
% 	\begin{equation*}
% 		c_1\mathcal M_k(u,Q^\nu_\rho(x))\le \F_{k}(u,Q^\nu_\rho(x)) \leq c_2\mathcal M_k(u,Q^\nu_\rho(x))
% 	\end{equation*}
% 	for every $u\in W^{1,p}(Q^\nu_\rho(x))$ with $0\leq u \leq 1$ a.e.\ in $Q^\nu_\rho(x)$. 
% 	Therefore to establish \eqref{c:bdds-f'-f''} 
% 	for every $x\in\R^n$ and $\rho>0$, where
% 	%\begin{equation*}%\label{c:min-bd}
% 	%\m_{k,\rho}(x,\nu) \defas \min\bigg\{\int_{Q^\nu_\rho(x)} \bigg(\frac{(1-v)^p}{\e_k}+\e_k^{p-1}|\nabla v|^p\bigg)\dy \colon v\in \Adm_{\e_k,\rho}(x,\nu)\bigg\}.
% 	%\end{equation*}
% 	\begin{equation*}%\label{c:min-bd}
% 		\m_{k,\rho}(x,\nu) \defas \min\bigg\{\mathcal M_k(u,Q^\nu_\rho(x))
% 		\colon u\in \Adm(\bar u_{x,\e_k}^{\nu},Q^{\nu}_\rho(x))
% 		\bigg\}.
% 	\end{equation*}
% 	Let $x\in\R^n$, $\nu\in\S^{n-1}$, $\rho>0$, and let $k\in\N$ be such that $2\e_k<\rho$. 
By the homogeneity and rotation invariance of the Modica-Mortola functional there holds 
	\[
	\m_{k,\rho}(x,\nu) = \m_{k,\rho}(0,e_n).
	\] 
	Let $\eta>0$ be arbitrary fixed. By the $\Gamma$-convergence result in \cite{MoMo} we can find a sequence $(u_k)\subset W^{1,p}(Q_\rho(0))$, $0\le u_k\le 1$ such that $u_k$ converges to $u_0^{e_n}$ in $L^p(Q_\rho(0))$ and 
		\begin{equation*}\label{c:cp-recovery}
		\limsup_{k\to+\infty}\mathcal M_k(u_k,Q_\rho(0))=(c_p+\eta)\rho^{n-1}\,.
	\end{equation*}
	Up to suitably modifying $u_k$ (for example by using Proposition \ref{prop:fund-est}), we can assume that $u_k=u_0^{e_n}$ in $\{|y_n|>\e_k \}$ so that in particular $(u_k)\subset \Adm(\bar u_{0,\e_k}^{e_n},Q_\rho(0))$. 
	In particular, by the arbitrariness of $\eta>0$ we deduce that
	\begin{equation}\label{c:cp-one}
	\limsup_{k\to +\infty}\m_{k,\rho}(x,\nu)= 		\limsup_{k\to +\infty}\m_{k,\rho}(0,e_n) \leq c_p \rho^{n-1}\,.
	\end{equation}
	On the other hand by Fubini's Theorem we have
	\begin{equation}\label{c:jg}
		\begin{split}
			%\int_{Q_\rho(0)} \bigg(\frac{W(u)}{\e_k}+\e_k^{p-1}|\nabla u|^p\bigg)\dx
		\mathcal	M_k(u,Q_\rho(0))
			&= \int_{Q'_{\rho}}\int_{-\frac{\rho}{2}}^{\frac{\rho}{2}} \bigg(\frac{W(u(x',x_n))}{\e_k}+\e_k^{p-1}|\nabla u (x',x_n)|^p\bigg)\dx_n \dx'\\
			&\geq \int_{Q'_{\rho}}\int_{-\frac{\rho}{2}}^{\frac{\rho}{2}} \bigg(\frac{W(u(x',x_n))}{\e_k}+\e_k^{p-1}\Big|\frac {\partial u(x',x_n)}{\partial x_n}\Big|^p\bigg)\dx_n \dx'\,,
		\end{split}
	\end{equation}
and by Young's inequality 
\begin{equation}\label{young}
\begin{split}
    \int_{-\frac{\rho}{2}}^{\frac{\rho}{2}}\Bigg(
	\frac{W(u(x',x_n))}{\e_k}&+\e_k^{p-1}\bigg|\frac {\partial u (x',x_n)}{\partial x_n}\bigg|^p\Bigg)\dx_n\\
	\ge
	&\Big(\frac{p}{p-1}\Big)^{\frac{p-1}{p}}p^{\frac{1}{p}}
	\int_{-\frac{\rho}{2}}^{\frac{\rho}{2}}\Bigg((W(u(x',x_n)))^{\frac{p-1}p}\bigg|\frac {\partial u (x',x_n)}{\partial x_n}\bigg|
	\Bigg)\dx_n\,,
\end{split}
\end{equation}
	for $\L^{n-1}$-a.e.\ $x'\in Q'_\rho$.
	Now if $u\in \Adm((\bar u_{0,\e_k}^{e_n},Q_\rho(0))$, then it coincides with $u_0^{e_n}$ in a neighbourhood of 
	\[
	\partial^\pm Q_\rho(0)\defas \Big\{(x',x_n)\in \R^{n-1}\times \R \colon x'\in \overline{Q'_\rho}, \; x_n=\pm\frac{\rho}{2}\Big\}\,.
	\]
	Therefore, for $\L^{n-1}$-a.e.\ $x'\in Q'_{\rho}$, the function $u_{x'}(t)\defas u(x',t)$ belongs to $W^{1,p}(-\frac{\rho}{2},\frac{\rho}{2})$ and satisfies 
	%
	%\begin{equation}\label{cond:1-dim-restriction}
	%v_{x'}(t)u_{x'}'(t)=v(x',t)\frac{\partial u(x',t)}{\partial x_n}\cdot e_1=0\;\text{ for $\L^1$-a.e.}\; t\in\Big(-\frac{\rho}{2},\frac{\rho}{2}\Big)\,,
	%\end{equation}
	%
	$u_{x'}(-\frac{\rho}{2})=0$, and $u_{x'}(\frac{\rho}{2})=1$,
 and a change of variable in \eqref{young} yields	\begin{align}\label{c:j-g}
		\int_{-\frac{\rho}{2}}^{\frac{\rho}{2}}\bigg(\frac{W(u(x',x_n))}{\e_k}+\e_k^{p-1}\bigg|\frac {\partial u (x',x_n)}{\partial x_n}\bigg|^p\bigg)\dx_n
		\geq \Big(\frac{p}{p-1}\Big)^{\frac{p-1}{p}}p^{\frac{1}{p}}\int_0^1 W(t)^{\frac{p-1}p}\dt=c_p\,.
	\end{align}
  Thus, combining \eqref{c:jg} with \eqref{c:j-g} we obtain
	\[
	%\int_{Q_\rho(0)} \bigg(\frac{(1-v)^p}{\e_k}+\e_k^{p-1}|\nabla v|^p\bigg)\dx 
	\mathcal M_k(u,Q_\rho(0))\geq c_p \rho^{n-1},
	\]
	for every $u\in \Adm((\bar u_{0,\e_k}^{e_n},Q_\rho(0))$. Passing to the infimum on $u$ and to the liminf as $k\to +\infty$ we get
	\begin{equation}\label{c:cp-two}
		\liminf_{k\to +\infty}\m_{k,\rho}(x,\nu)=
		\liminf_{k\to +\infty}\m_{k,\rho}(0,e_n) \geq c_p \rho^{n-1},
	\end{equation}
	for every $\rho>0$. Eventually, gathering \eqref{c:cp-one} and \eqref{c:cp-two} we get \eqref{c:claim}, and hence the thesis. 
\end{proof}
We now have all the ingredients to prove the main result of this paper, namely, Theorem \ref{thm:main-result}. 

\begin{proof}[Proof of Theorem \ref{thm:main-result}] The proof follows by combining Theorem \ref{thm:int-rep},  Proposition \ref{p:surface-term} and Proposition \ref{prop:f'-f''}.
\end{proof}

%%%%%%%%%%%%%%%%%%%%%%%%%%%%%%%%%%%
%
%STOCHASTIC HOMOGENISATION
%
%%%%%%%%%%%%%%%%%%%%%%%%%%%%%%%%%%%%%
\section{Stochastic homogenisation}\label{sect:stochastic-homogenisation}
\noindent
In this section we derive a $\Gamma$-convergence result for functionals of type $\F_k$ when $f_k$ are random integrands of the form
\[
 f_k(\om, x,u,\xi)=f\Big(\omega,\frac{x}{\e_k},u,\xi\Big),
\]
where $\om$ belongs to the sample space $\Omega$ of a complete  probability space $(\Omega,\T,P)$. 

\medskip

In order to do that we need to recall some useful definitions.

\begin{definition}[Group of $P$-preserving transformations] Let $d\in \N$, $d\geq 1$. A group of $P$-preserving transformations on $(\Omega,\T,P)$  is a family $(\tau_z)_{z\in\Z^d}$ of mappings $\tau_z\colon\Omega\to\Omega$ satisfying the following properties:
	\begin{enumerate}[label=(\arabic*)]
		\item (measurability) $\tau_z$ is $\T$-measurable for every $z\in\Z^d$;
		\item (invariance) $P(\tau_z(E))=P(E)$, for every $E\in\T$ and every $z\in\Z^d$;
		\item (group property) $\tau_0={\rm id}_\Omega$ and $\tau_{z+z'}=\tau_z\circ\tau_{z'}$ for every $z,z'\in\Z^d$.
	\end{enumerate}
	If, in addition, every $(\tau_z)_{z\in\Z^d}$-invariant set (\textit{i.e.}, every $E\in\T$ with  
	$\tau_z(E)=E$ for every $z\in\Z^d$) has probability 0 or 1, then $(\tau_z)_{z\in\Z^d}$ is called ergodic.
\end{definition}
Let $a\defas(a_1,\dots,a_d),\,b\defas(b_1,\dots,b_d)\in\Z^d$ with $a_i<b_i$ for all $i\in\{1,\ldots,d\}$; we define the $d$-dimensional interval 
\begin{equation*}
[a,b):=\{x\in\Z^d\colon a_i\le x_i<b_i\,\,\text{for}\,\,i=1,\ldots,d\}
\end{equation*}
and we set
\begin{equation*}
\I_d:=\{[a,b)\colon a,b\in\Z^d\,,\, a_i<b_i\,\,\text{for}\,\,i=1,\ldots,d\}\,.
\end{equation*}
\begin{definition}[Subadditive process]\label{sub_proc}
	A discrete subadditive process with respect to a group $(\tau_z)_{z\in\Z^d}$ of $P$-preserving transformations on $(\Omega,\T,P)$ is a function $\mu\colon\Omega \x \I_d\to\R$ satisfying the following properties:
	\begin{enumerate}[label=(\arabic*)]
		\item\label{subad:meas} (measurability) for every $A\in\I_d$ the function $\omega\mapsto\mu(\omega,A)$ is $\T$-measurable;
		\item\label{subad:cov} (covariance) for every $\omega\in\Omega$, $A\in\I_d$, and $z\in\Z^d$ we have $\mu(\omega,A+z)=\mu(\tau_z(\omega),A)$;
		\item\label{subad:sub} (subadditivity) for every $A\in\I_d$ and for every finite family $(A_i)_{i\in I}\subset\I_d$ of pairwise disjoint sets such that $A=\cup_{i\in I}A_i$, we have
		\begin{equation*}
		\mu(\omega,A)\le\sum_{i\in I}\mu(\omega,A_i)\quad\text{for every}\,\,\omega\in\Omega\,;
		\end{equation*}
		\item\label{subad:bound} (boundedness) there exists $c>0$ such that $0\le\mu(\omega,A)\le c\L^d(A)$ for every $\omega\in\Omega$ and $A\in\I_d$.
	\end{enumerate}
\end{definition}

\begin{definition}[Stationarity] \label{def:stationarity}
Let $(\tau_z)_{z\in\Z^n}$ be a group of $P$-preserving transformations on $(\Omega,\T,P)$. 
We  say that $f\colon\Omega\x\R^n\x\R \x\R^n\to[0,+\infty)$ is stationary with respect to $(\tau_z)_{z\in\Z^n}$ if
	\begin{equation*}
f(\omega,x+z,u,\xi)=f(\tau_z(\omega),x,u,\xi)
	\end{equation*}
	for every $\omega\in\Omega$, $x\in\R^n$, $z\in\Z^n$, $u\in[0,1]$ and $\xi\in\R^n$. 	
\end{definition}

We consider random integrands $f\colon\Omega\x\R^n\x\R\x\R^n\to[0,+\infty)$ satisfying the following assumptions:

\smallskip

\begin{enumerate}[label= ($F\arabic*$)]
	\item\label{meas_f} $f$ is $(\T\otimes\mathcal{B}^n\otimes\mathcal{B}\otimes\mathcal{B}^n)$-measurable;
	\item\label{ginF} $f(\omega,\,\cdot\,,\,\cdot\,,\,\cdot\,)\in\mathcal F$ for every $\omega\in\Omega$;
	\item\label{cont_in_xi} 
	For every $\omega\in\Omega$, and every $x\in\R^n$ the map $(u,\xi)\mapsto f(\omega,x,u,\xi)$ is lower semicontinuous.
\end{enumerate}

\medskip

\noindent 
For any $f$ as above we define the sequence of random phase-field functionals $\F_k(\om) \colon  L^1_\loc(\R^n) \times \A \longrightarrow [0,+\infty]$ given by
\begin{equation}\label{F_e_omega}
\F_k(\om)(u, A)\defas \frac{1}{\e_k}\int_A f\left(\omega,\frac{x}{\e_k},u,\e_k\nabla u\right)\dx\,,
\end{equation}
if $u\in  W^{1,p}(A)$, $0\leq u\leq 1$ and extended to $+\infty$ otherwise.  
For $\omega\in\Omega$, $A\in\A$  we define
\begin{equation*}\label{F-omega}
	\F(\omega)(u,A)\defas\int_Af(\omega,x,u,\nabla u)\dx\,,
\end{equation*}
and
\begin{equation}\label{m-omega}
	\m_\omega(\bar u^\nu_z,A):=\inf\Bigl\{	\F(\omega)(u,A)\colon u\in \Adm(\bar u^\nu_z,A)\Bigr\}\,,
\end{equation}
where $\Adm(\bar u^\nu_z,A)$ is as in \eqref{c:adm-e}, with $\bar u^\nu_{x,\e_k}$ replaced by $\bar u^\nu_z$\ie with $\e_k=1$.
Moreover for every $A\subset\R^n$ with $\rm int\, A\in\A$ we set $$\m_{\om}(\bar{u}_x^\nu,A)\defas\m_{\om}(\bar{u}_x^\nu,\rm int\, A),\quad \text{and}\quad \Adm(\bar{u}_0^\nu,A)\defas \Adm(\bar{u}_0^\nu,\rm int\, A).$$

\medskip

Finally we are able to state the main result of this section.

\begin{theorem}[Stochastic homogenisation]\label{thm:stoch_hom_2}
Let $f$  be a random integrand satisfying \ref{meas_f}-\ref{cont_in_xi}. Assume also that $f$ is stationary with respect to a group $(\tau_z)_{z\in\Z^n}$ of $P$-preserving transformations on $(\O,\mathcal T,P)$.
For every $\omega\in\Omega$ let $\F_k(\omega)$ be as in \eqref{F_e_omega} and $\m_\omega$ be as in~\eqref{m-omega}.  Then there exists $\Omega'\in\T$, with $P(\Omega')=1$ such that for every $\omega\in\Omega'$, $x\in\R^n$, $\nu\in \Sph^{n-1}$ the limit
	\begin{equation}\label{eq:f-stoch-hom}
	\lim_{r\to +\infty} \frac{\m_\omega(\bar u^\nu_{rx},Q^\nu_r(rx))}{r^{n-1}}=
	\lim_{r\to +\infty} \frac{\m_\omega(\bar u^\nu_{0},Q^\nu_r(0))}{r^{n-1}}=:f_\hom(\om,\nu)
	\end{equation}
exists and is independent of $x$. The function  $f_\hom \colon \Omega \times \Sph^{n-1} \to [0,+\infty)$ is $(\T\otimes \mathcal B(\Sph^{n-1}))$-measurable.

Moreover, for every $\omega\in\Omega'$ and for every $A\in \A$ the functionals $\F_k(\omega)(\cdot,A)$ $\Gamma$-converge in $ L^1_\loc(\R^n)$ to the functional $\F_{\hom}(\omega)(\cdot,A)$ with $\F_{\hom}(\omega)\colon  L^1_\loc(\R^n)\times\A\longrightarrow[0,+\infty]$ given by
\begin{equation*}
\F_{\rm hom}(\omega)(u,A)\defas
\begin{cases}
\displaystyle\int_{S_u \cap A} f_{\rm hom}(\omega,\nu_u)\dHn &\text{if}\ u \in BV(A;\{0,1\})\, \\[4pt]
+\infty &\text{otherwise}\,.
\end{cases}
\end{equation*}
If, in addition, $(\tau_z)_{z\in\Z^n}$ is ergodic, then $f_{\hom}$ is independent of $\omega$ and 
		\begin{align*}
f_{\rm hom}(\nu) &=
	\lim_{r\to +\infty} \frac{1}{r^{n-1}}\int_\Omega \m_\omega(\bar u^\nu_{0},Q^\nu_r(0))\dP(\omega)\,,
	\end{align*}
	thus, $\F_\hom$ is deterministic.
\end{theorem}

\begin{remark} The notion of stationarity  given in Definition \ref{def:stationarity} generalises the notion of spatial periodicity.
Therefore  Theorem \ref{thm:stoch_hom_2} applies also to the case of deterministic periodic homogenisation. However we will show later (see Section \ref{sect:periodic-hom}) that condition \ref{cont_in_xi} is not needed if we restrict to the periodic setting.
\end{remark}
The rest of the section is dedicated to prove Theorem \ref{thm:stoch_hom_2}.

\subsection{Existence of the limit}\label{sect:cell-formulas}
The almost sure $\Gamma$-convergence result in Theorem~\ref{thm:stoch_hom_2} readily follows by Theorem~\ref{thm:hom} if there exists a $\T$-measurable set $\Omega' \subset \Omega$, with $P(\Omega')=1$, such that for every $\om \in \Omega'$ the limit in~\eqref{eq:f-stoch-hom} exists and is independent of $x$, and if $f_\hom$ is $(\T\otimes \mathcal B(\Sph^{n-1}))$-measurable. We will prove that this is actually the case when the integrand $f$ satisfies \ref{meas_f}-\ref{cont_in_xi} and is stationary in the sense of Definition \ref{def:stationarity}.

\medskip

Since the boundary conditions given in the definition of $\m_\om(\bar u_{rx}^\nu,Q_r^\nu(rx))$ depend on $x$, the proof of the   
existence and $x$-homogeneity of the limit in~\eqref{eq:f-stoch-hom} does not follow by a direct application of the Subadditive Ergodic Theorem ~\cite[Theorem 2.4]{AK81}. For this reason we need to follow a more technical argument, in the same way as in \cite{ACR11, CDMSZ19a, BCR18},  which can be divided into three main steps. In the first step we prove that when $x=0$ the minimisation problem \eqref{m-omega} defines a subadditive process on $\O \times \mathcal I_{n-1}$ (see Proposition \ref{prop:subadditive}). 
In the second step we prove the almost sure existence of  the limit in \eqref{eq:f-stoch-hom} when $x=0$ (see Proposition \ref{prop:ex-limit-zero}). Eventually in the third one we show that the same holds for an arbitrary $x\in\R^n$ (see Proposition \ref{prop:ex-limit-x}).\\

For $\nu\in \Sph^{n-1}$ we let $R_\nu$ be an orthogonal matrix as in~\ref{Rn}. 
Note that $\{R_\nu e_i\colon i=1\,,\ldots\,,n-1\}$ is an orthonormal basis for $\Pi^\nu$, and $R_\nu\in\Q^{n\x n}$, if $\nu\in\Sph^{n-1}\cap\Q^n$. Let $M_\nu>2$ be an integer such that $M_\nu R_\nu\in\Z^{n\x n}$; therefore $M_\nu R_\nu (z',0)\in\Pi^\nu\cap\Z^n$ for every $z'\in\Z^{n-1}$. 

Let $I\in\mathcal{I}_{n-1}$\ie $I=[a,b)$ with $a,b\in\Z^{n-1}$. For every $\nu\in\S^{n-1}$ and every $I\in\mathcal{I}_{n-1}$ we define the $n$-dimensional interval $I_\nu$ as
\begin{equation}\label{def:interval-n-dim}
I_\nu\defas M_\nu R_\nu\big(I\x[-c,c)\big)\quad\text{where}\quad c\defas\frac{1}{2}\max_{i=1,\ldots,n-1}(b_i-a_i)\,.
\end{equation}
Eventually, for fixed $\nu \in \Sph^{n-1}\cap \Q^n$ we consider the function $\mu_\nu \colon \O \times \mathcal I_{n-1} \to \R$ given by
\begin{equation}\label{def:subad-process-g}
\mu_\nu(\omega,I)\defas\frac{1}{M_\nu^{n-1}}\m_\om(\bar{u}_{0}^\nu, I_\nu)\,,
\end{equation} 
where $\m_\om (\bar{u}_{0}^\nu, I_\nu)$ is as in \eqref{m-omega} with $x=0$ and $A=I_\nu$. 
%For later convenience we observe that by \eqref{hyp:ub-f} 
%\begin{equation}\label{laterconvenience}
%\m_\om (\bar{u}_{0}^\nu, I_\nu)\le \F(\om)(\bar{u}_{0}^\nu, I_\nu)
%\le c_2C_\uu\L^n(I)\,.
%\end{equation}
Then the following result holds true.

%We next show that $\mu_\nu$ defines a subadditive process on $\O \times \mathcal I_{n-1}$.

\begin{proposition}\label{prop:subadditive}
Let $f$ satisfy \ref{meas_f}-\ref{cont_in_xi} and assume that it is stationary with respect to a group $(\tau_z)_{z\in\Z^n}$ of $P$-preserving transformations on $(\Omega,\T,P)$. Let $\nu\in\Sph^{n-1}\cap \Q^n$ and let $\mu_\nu\colon \O \times \mathcal I_{n-1} \mapsto \R$ be as in \eqref{def:subad-process-g}.
Then there exists a group of $P$-preserving transformations $(\tau_{z'}^\nu)_{z'\in\Z^{n-1}}$ on $(\Omega,\T,P)$ such that $\mu_\nu$ is a subadditive process on $(\Omega,\T,P)$ with respect to $(\tau_{z'}^\nu)_{z'\in\Z^{n-1}}$. Moreover, it holds
\begin{equation}\label{c:bd-mu}
0\le\mu_\nu(\omega, I)\le c_2 C_\uu \L^{n-1}(I),
\end{equation}
for $P$-a.e.\ $\om \in \Omega$ and for every $I\in \mathcal I_{n-1}$.  
\end{proposition}
\begin{proof}
Let $\nu\in\S^{n-1}\cap\Q^n$ be fixed;  below we show that there exists a group of $P$-preserving transformations $(\tau_{z'}^\nu)_{z'\in\Z^{n-1}}$ for which  $\mu_\nu$ satisfies conditions~\ref{subad:meas}--\ref{subad:bound} of Definition \ref{sub_proc}.

%We divide the proof into four steps, each of them corresponding to one of the four conditions in Definition \ref{sub_proc}.  

\medskip

\textit{Step 1: measurability.} The measurabilty can be achieved by suitably adapting the proof of \cite[Lemma C.1.]{RufRuf}, for this reason here we only sketch it.
Let $I\in \mathcal I_{n-1}$ and let $I_\nu \subset \R^n$ be as in  \eqref{def:interval-n-dim}. 
For every $n\in\N$ we let $f^n$ be the Moreau-Yosida-regularisation of $f$ with respect to the last two variables, that is 
\begin{equation*}
	f^n(\omega,x,u,\xi)\defas\inf_{(v,\zeta)\in\R\times\R^n}\left\{f(\omega,x,v,\zeta)+n|(v,\zeta)-(u,\xi)|\right\}\,.
\end{equation*}
It is known that $f^n$ is $n$-Lipschitz in  $(u,\xi)$.  Then arguing as in the proof of \cite[Lemma C.1.]{RufRuf} it can be shown that the map $\F^n(\omega)\colon W^{1,p}(I_\nu)\to [0,+\infty)$  given by
\begin{equation*}
\F^n(\omega)(u)\defas \int_{I_\nu}f^n(\omega,x,u,\nabla u)\dx\,,
\end{equation*}
 is well defined for every $\omega\in\O$ and that $(\omega,u)\mapsto\F^n(\omega)(u)$ is
$\T\otimes \mathcal B(W^{1,p}(I_\nu))$-measurable. By \ref{cont_in_xi}  $f^n\nearrow f$ pointwise so that in particular $\F^n(\omega)(u)$ converges to $\F(\omega)(u,I_\nu)$ pointwise. This in turn implies that $(\omega,u)\mapsto\F(\omega)(u,I_\nu)$ is $\T\otimes \mathcal B(W^{1,p}(I_\nu))$-measurable as well.
Now recall that $(\O,\T,P)$ is a complete probability space, while  the set 
\begin{equation*}
\Adm(\bar{u}_0^\nu,I_\nu)= \{u\in W^{1,p}(I_\nu)\, ,\; 0\leq u \leq 1\,, \ u= \bar u^\nu_0\ \text{ near }\ \partial I_\nu 
\}\subset W^{1,p}(I_\nu)
\end{equation*}
defines a separable metric space, when endowed with the distance induced by the $W^{1,p}(I_\nu)$-norm. Moreover from \ref{cont_in_xi} the map $u\mapsto\F(\omega)(u,I_\nu)$ is lower semicontinuous and not constantly $+\infty$ since from \eqref{1dim-energy-bis} we have 
$
\F(\omega)(\bar u_0^\nu,I_\nu)<+\infty
$.
Hence we can apply \cite[Lemma C.2.]{RufRuf} to deduce the $\T$-measurability of $\om \mapsto \m_\om(\bar u^\nu_0,I_\nu)$ and in particular of $\om \mapsto \mu_\nu(\omega,I)$.  

\medskip

\textit{Step 2: covariance.} Let $z'\in\Z^{n-1}$ be fixed and let $I\in\I_{n-1}$.  Note that \eqref{def:interval-n-dim} implies the following equality
\begin{equation*}
(I+z')_\nu=I_\nu +M_\nu R_\nu(z',0)=I_\nu+z'_\nu\,,
\end{equation*}
where $z'_{\nu}:=M_\nu R_\nu(z',0) \in \Z^n\cap\Pi^\nu$.
Therefore we have
\begin{equation}\label{cov_1}
	\mu_\nu(\omega,I+z')=\frac{1}{M_\nu^{n-1}}\m_\omega(\bar u_0^\nu,I_\nu+z'_\nu)\,.
\end{equation}
Let $u\in \Adm(\bar{u}_0^\nu,I_\nu+z_\nu')$ and define $\tilde u(x)\defas u(x+z'_\nu)$.
Clearly $\tilde{u}\in\Adm(\bar{u}_0^\nu, I_\nu)$. Indeed since $z_\nu'\in\Pi^\nu$, we have 
$\tilde{u}=\bar{u}_0^\nu(\cdot+z'_\nu)=\bar{u}_0^\nu$
near $\partial I_\nu$. Further by  a change of variables and using the stationarity of $f$ we obtain the following 
\begin{equation*}
	\begin{split}
	\F(\omega)(u,{\rm int}\, (I_\nu+z'_\nu)) =&\int_{I_\nu+z'_\nu}f(\omega,x,u,\nabla u)\dx
	=\int_{I_\nu}f(\omega,x+z'_\nu,\tilde u,\nabla \tilde u)\dx\\=&\int_{I_\nu}f(\tau_{z'_\nu}(\omega),x,\tilde u,\nabla\tilde u)\dx=\F(\tau_{z'_\nu}(\omega))(\tilde u, {\rm int}\, I_\nu)\,.
	\end{split}
\end{equation*}
 Let  $(\tau^\nu_{z'})_{z'\in\Z^{n-1}}:=(\tau_{z'_\nu})_{z'\in\Z^{n-1}}$; then  $(\tau^\nu_{z'})_{z'\in\Z^{n-1}}$ is well defined since $z'_\nu\in\Z^n$ and it defines a group of $P$-preserving transformations on $(\Omega,\T,P)$. In particular the equality above becomes
\begin{equation}\label{cov_2}
	\F(\omega)(u, {\rm int}\,(I_\nu+z'_\nu))=\F(\tau_{z'}^\nu(\omega))(\tilde u,{\rm int}\, I_\nu)\,.
\end{equation}
%Moreover, since $z_\nu'\in\Pi^\nu$, we have 
%%
%\[
%\tilde{u}=\bar{u}_0^\nu(\cdot+z'_\nu)=\bar{u}_0^\nu
%\]
%% 
%near $\partial I_\nu$\ie $\tilde{u}\in\Adm(\bar{u}_0^\nu, I_\nu)$.
Eventually, gathering \eqref{cov_1} and \eqref{cov_2}, by the arbitrariness of $u$ we infer 
\begin{equation*}
\mu_\nu(\omega,I+z')=
\mu_\nu(\tau^\nu_{z'}(\omega),I)\,,
\end{equation*}
and the covariance of $\mu_\nu$ with respect to $(\tau^\nu_{z'})_{z'\in\Z^{n-1}}$ is shown.

\medskip

\textit{Step 3: subadditivity.} Let  $\omega\in\Omega$, $I\in\I_{n-1}$, and let $\{I_1,\dots,I_N\}\subset\I_{n-1}$ be a finite family of pairwise disjoint sets such that $I=\bigcup_i I_i$.
For $\eta>0$ fixed and for every $i=1,\dots,N$ we let $u_i\in W^{1,p}((I_i)_\nu)$ be admissible for $\m_\omega(\bar u_0^\nu,(I_i)_\nu)$ such that 
\begin{equation}\label{sub_1}
	\F(\omega)(u_i,{\rm int}\,(I_i)_\nu)\le\m_\omega(\bar u_0^\nu,(I_i)_\nu)+\eta\,. 
\end{equation}
We let
\begin{equation*}
u:=\begin{cases}
u_i&\text{in }\;(I_i)_\nu\,,\ i=1\,,\ldots,N\,,\\
\bar u_0^\nu&\text{in }\; I_\nu\setminus\bigcup_i (I_i)_\nu\,,
\end{cases}
\end{equation*}
in this way $u\in  W^{1,p}(I_\nu;\R^m)$ and $u=\bar u_0^\nu$ near $\partial I_\nu$.
Hence in particular $u\in\Adm(\bar{u}_0^\nu,I_\nu)$ and  
\begin{equation}\label{est:subad}
\F(\omega)(u,{\rm int}\,I_\nu)=\sum_{i=1}^N\F(\omega)(u_i,{\rm int}\,(I_i)_\nu)+\F(\omega)\big(\bar u^\nu_0,{\rm int}\,(I_\nu\setminus\bigcup_{i=1}^N(I_i)_\nu)\big)\,.
\end{equation}
We observe that 
\begin{equation}\label{f=0}
\F(\omega)\big(\bar u^\nu_0,{\rm int}\,(I_\nu\setminus\bigcup_{i=1}^N(I_i)_\nu)\big)=0\,.
\end{equation}
Indeed, being $M_\nu>2$ and $c\geq\frac{1}{2}$ in~\eqref{def:interval-n-dim} we have that
$\{y\in I_\nu\colon|y\cdot\nu|\le 1\}\subset\bigcup_{i}(I_i)_\nu$ and thus
 $\bar u_0^\nu\equiv u_0^\nu$ in $I_\nu\setminus\bigcup_i(I_i)_\nu$. Now recalling~\eqref{f-value-at-0} we derive \eqref{f=0}.
Putting together ~\eqref{sub_1}--\eqref{f=0} we obtain
\begin{equation*}
\m_\omega(\bar u_0^\nu,I_\nu)\le \F(\omega)(u,{\rm int}\,I_\nu)=\sum_{i=1}^N\F(\omega)(u_i,{\rm int}\,(I_i)_\nu)\le \sum_{i=1}^N\m_\omega(\bar u_0^\nu,(I_i)_\nu)+N\eta\,,
\end{equation*}
and the subadditivity of $\mu_\nu$ follows by letting $\eta\to0$-

\medskip

\textit{Step 4: boundedness.} Let $\omega\in \Omega$ and $I\in\I_{n-1}$. From \eqref{def:interval-n-dim} and \eqref{1dim-energy-bis} we easily deduce
\begin{equation*}
0\le\mu_\nu(\omega, I)=\frac{1}{M_\nu^{n-1}}\m_\omega(\bar u_0^\nu,I_\nu)\le c_2 C_\uu\L^{n-1}(I).
\end{equation*}
\end{proof}
 Proposition~\ref{prop:subadditive} is the key ingredient to prove that almost surely the limit defining $f_\hom$ exists when $x=0$.

\begin{proposition}[Homogenised surface integrand for $x=0$]\label{prop:ex-limit-zero}
Let $f$ satisfy \ref{meas_f}-\ref{cont_in_xi} and assume that it is stationary with respect to a group $(\tau_z)_{z\in\Z^n}$ of $P$-preserving transformations on $(\Omega,\T,P)$. For $\omega\in\Omega$ let $\m_\omega$ be as in~\eqref{m-omega}. Then there exist $\tilde{\Omega}\in\mathcal T$ with $P(\tilde{\Omega})=1$ and a $(\T\otimes \mathcal B(\Sph^{n-1}))$-measurable function $f_{\hom}:\Omega\x\Sph^{n-1}\to[0,+\infty)$ such that
\begin{equation}\label{eq:ex-limit-zero}
\lim_{r\to+\infty}\frac{\m_\om(\bar{u}_{0}^\nu,Q_{r}^\nu(0))}{r^{n-1}}=f_\hom(\om,\nu)
\end{equation}
for every $\omega\in\tilde{\Omega}$ and every $\nu\in\Sph^{n-1}$. Moreover, $\tilde \Omega$ and $f_\hom$ are $(\tau_z)_{z\in \Z^n}$-translation invariant\ie $\tau_{z}(\tilde{\Omega})=\tilde{\Omega}$ for every $z\in \Z^n$ and
\begin{equation}\label{eq:shift-invariance}
f_\hom(\tau_z(\omega),\nu)=f_\hom(\omega,\nu),
\end{equation}
for every $z\in \Z^n$, for every $\omega\in\tilde{\Omega}$, and every $\nu\in\Sph^{n-1}$. Eventually, if $(\tau_z)_{z\in\Z^n}$ is ergodic then $f_{\hom}$ is independent of $\omega$ and given by
\begin{equation}\label{c:f-hom-er}
f_{\rm hom}(\nu)=
\lim_{r\to +\infty} \frac{1}{r^{n-1}}\int_\Omega \m_\omega(\bar u^\nu_0,Q^\nu_r(0))\dP(\omega).
\end{equation}
\end{proposition}
\begin{proof}
We divide the proof into three steps.  

\medskip

\textit{Step 1: 
existence of the limit for $\nu\in\Sph^{n-1}\cap\Q^n$.
}
In this step we show that there exists $\tilde\O\in\T$ with $P(\tilde\O)=1$ such that the following holds: for  every $\nu\in\Sph^{n-1}\cap\Q^n$ there is a $\mathcal T$-measurable function $f_\nu \colon \Omega \to [0,+\infty)$ such that
	\begin{equation*}
f_\nu (\om)=
 \lim_{r\to+\infty}\frac{\m_\om(\bar{u}_{0}^\nu,Q_{r}^\nu(0))}{r^{n-1}}\,, \quad \forall \om \in \tilde\Omega\,.
\end{equation*}

We fix $\nu\in\Sph^{n-1}\cap\Q^n$.  By Proposition~\ref{prop:subadditive} we know that the map $\mu_\nu$ given in \eqref{def:subad-process-g} is a subadditive process on $(\Omega,\T,P)$ with respect to some group of $P$-preserving transformations. This allows us to apply the Subadditive Ergodic  Theorem \cite[Theorem 2.4]{AK81} and deduce the following: let $I=[-1,1)^{n-1}$, so that $I_\nu=2M_\nu Q^\nu(0)$; then there exists a set $\Omega_\nu \in \mathcal T$, with $P(\Omega_\nu)=1$, and a $\mathcal T$-measurable function $f_\nu \colon \Omega \to [0,+\infty)$ such that
	\begin{equation}\label{erg-thm}
f_\nu (\om)=
\lim_{j\to +\infty}\frac{{\mu_\nu}(\om, j I)}{(2j)^{n-1}} = 
\lim_{j\to +\infty}\frac{\m_\om(\bar u_0^\nu, j2M_\nu Q^\nu(0))}{(j2M_\nu)^{n-1}}\,. 
\end{equation}
for every $\om \in \Omega_\nu$. \\

Let now $(r_j)$ be a sequence of strictly positive real numbers with $r_j \to +\infty$, as $j \to +\infty$ and define
\begin{equation*}
r_j^-\defas 2M_\nu\Big(\Big\lfloor \frac{r_j}{2M_\nu}\Big\rfloor-1\Big)\quad\text{and}\quad r_j^+\defas 2M_\nu\Big(\Big\lfloor \frac{r_j}{2M_\nu}\Big\rfloor+2\Big)\,.
\end{equation*}
For $j$ sufficiently large $r_j>4(1+M_\nu)$, and thus $r_j^->4$, moreover
\[  
Q^\nu_{r_j^-+2}(0)\wcont Q_{r_j}^\nu(0)\wcont Q_{r_j+2}^\nu(0)\wcont
Q_{r_j^+}^\nu(0)\,.
\]
We then can apply Lemma~\ref{lem:cubes-shift} twice: the first time
with $x=\tilde x=0$, $r=r_j^-$, and $\tilde r=r_j$ and the second one with $x=\tilde x=0$, $r=r_j^+$, and $\tilde r=r_j$ and get the two following estimates
\begin{equation}\label{up}
\frac{\m_\omega(\bar{u}_0^\nu,Q_{r_j}^\nu(0))}{r_j^{n-1}}\le 
\frac{\m_\omega(\bar{u}_{0}^\nu,Q_{r_j^-}^\nu(0))}{(r_j^-)^{n-1}}
+
\frac{L(r_j-r_j^-+1)}{r_j}
\,,
\end{equation}
\begin{equation}\label{low}
\frac{\m_\omega(\bar{u}_{0}^\nu,Q_{r_j}^\nu(0))}{r_j^{n-1}}\ge
\frac{\m_\omega(\bar{u}_0^\nu,Q_{r_j^+}^\nu(0))}{(r_j^+)^{n-1}}-
\frac{L(r_j^+-r_j+1)}{r_j}\,.
\end{equation} 
Using that $r_j^+-r_j\leq 4M_\nu$ and $r_j-r_j^-\leq 4M_\nu$,
thus thanks to \eqref{erg-thm}, passing to the limsup in \eqref{up} and to the liminf  in \eqref{low} as $j \to +\infty$  yield \begin{equation}\label{c:ls-real}
\limsup_{j \to +\infty} \frac{\m_\omega(\bar{u}_{0}^\nu,Q_{r_j}^\nu(0))}{r_j^{n-1}} \leq f_\nu(\om)\quad\forall\om \in \Omega_\nu,
\end{equation}
and
\begin{equation}\label{c:li-real}
\liminf_{j \to +\infty} \frac{\m_\omega(\bar{u}_{0}^\nu,Q_{r_j}^\nu(0))}{r_j^{n-1}} \geq f_\nu(\om) \quad\forall\om \in \Omega_\nu. 
\end{equation}
From \eqref{c:ls-real} and \eqref{c:li-real} we have that for every $\om \in \Omega_\nu$ the limit along $(r_j)$ exists and satisfies
\begin{equation*}
\lim_{j \to +\infty} \frac{\m_\omega(\bar{u}_{0}^\nu,Q_{r_j}^\nu(0))}{r_j^{n-1}}=f_\nu(\om).
\end{equation*}
Eventually to conclude we define
\[
\tilde\Omega:=\hspace{-0.3cm}\bigcap_{\nu\in\S^{n-1}\cap\Q^n}\hspace{-0.3cm}\Omega_\nu,
\]  
which satisfies the desired properties.

\medskip

\textit{Step 2: existence of the limit for $\nu\in\Sph^{n-1}\sm\Q^n$.} In this step we prove that there is a $(\T\otimes \mathcal B(\Sph^{n-1}))$-measurable function $f_{\hom}:\Omega\x\Sph^{n-1}\to[0,+\infty)$ such that \eqref{eq:ex-limit-zero} holds for every $\omega\in\tilde{\Omega}$ and every $\nu\in\Sph^{n-1}$.

Let $\underline{f}, \overline{f} \colon \tilde \Omega \times \Sph^{n-1} \to [0,+\infty]$ be given by
\begin{equation*}
\underline{f}(\omega,\nu)\defas\liminf_{r\to+\infty}\frac{\m_\omega(\bar{u}_0^\nu,Q_r^\nu(0))}{r^{n-1}}\,,\quad \overline{f}(\omega,\nu)\defas\limsup_{r\to+\infty}\frac{\m_\omega(\bar{u}_0^\nu,Q_r^\nu(0))}{r^{n-1}}.
\end{equation*}
We first observe that $\widehat{\Sph}_\pm^{n-1}\cap\Q^n$ is dense in $\widehat{\Sph}_\pm^{n-1}$.
Moreover in Step 1 we showed that $\underline{f}(\omega,\nu)=\overline{f}(\omega,\nu)=f_\nu(\om)$, for every $\om \in \tilde \Omega$ and for every $\nu \in \Sph^{n-1}\cap\Q^n$. 
Therefore, in order to obtain the same equality for every for every $\omega\in\tilde{\Omega}$ and every $\nu \in \Sph^{n-1}$, it is enough to show that
 the restrictions of the functions $\nu\mapsto\underline f(\omega,\nu)$ and $\nu\mapsto\overline f(\omega,\nu)$ to the sets $\widehat{\S}_\pm^{n-1}$ are continuous. Further once we have that we can also deduce the following:
\[
\om \mapsto \overline f (\om,\nu) \; \text{ is $\mathcal T$-measurable in $\tilde \Omega$, for every $\nu \in \Sph^{n-1}$} 
\]
together with
\[
\nu \mapsto \overline f (\om,\nu) \; \text{ is continuous in $\widehat \Sph^{n-1}_\pm$, for every $\om \in \tilde \Omega$}\,,
\]
readily imply that the restriction of $\overline f$ to $\tilde \Omega \times \widehat \Sph^{n-1}_\pm$ is measurable with respect to the $\sigma$-algebra induced in $\tilde \Omega \times \widehat \Sph^{n-1}_\pm$ by $\mathcal T \otimes \mathcal B(\Sph^{n-1})$. 
Hence the claim  follows by setting 
\begin{equation}\label{c:def-f-hom}
f_{\hom}(\om,\nu):=\begin{cases} \overline f(\om,\nu) & \text{if }\; \om \in \tilde \Omega, 
\cr
c_2 c_p & \text{if }\; \om \in \Omega \setminus \tilde \Omega.
\end{cases}
\end{equation}

We now prove that $\overline f(\omega,\cdot)$ is continuous in $\widehat\S^{n-1}_+$. The other proofs are  analogous and therefore are left to the readers. 
Let $\nu\in\widehat{\S}_+^{n-1}$, $(\nu_j)\subset\widehat{\S}_+^{n-1}$ be such that $\nu_j\to\nu$, as $j \to +\infty$. Then, for every $\alpha\in (0,\frac{1}{2})$ we find $j_\alpha \in \N$ such that \eqref{c:30-condition} holds true for every $j\geq j_\alpha$. Hence we can apply Lemma~\ref{lem:cubes-rotation} with $x=0$ and $\tilde \nu =\nu_j$ and get
\begin{equation*}
\m_\om(\bar u^{\nu_j}_0, Q^{\nu_j}_{(1+\alpha)r}(0)) -c_\alpha r^{n-1} \leq \m_\om(\bar u^{\nu}_0, Q^{\nu}_{r}(0)) \leq \m_\om(\bar u^{\nu_j}_0, Q^{\nu_j}_{(1-\alpha)r}(0)) +c_\alpha r^{n-1},
\end{equation*}
where $c_\alpha \to 0$, as $\alpha \to 0$. 
Now dividing the above inequality by $r^{n-1}$ and passing to the limsup as $r \to +\infty$ we get 
 \begin{align}\label{c:eins}
(1+\alpha)^{n-1}\,\overline{f}(\omega,\nu_j) 
&\le \overline f(\omega,\nu)+c_\alpha\,,\\\label{c:zwei}
(1-\alpha)^{n-1}\,\overline{f}(\omega,\nu_j)&\geq \overline f(\omega,\nu)-c_\alpha\,.
 \end{align}
Eventually passing to the limsup as $j\to+\infty$ in \eqref{c:eins} and to the liminf as $j\to+\infty$ in \eqref{c:zwei}, and letting $\alpha\to0$ we have
\begin{equation*}
	\limsup_{j\to+\infty}\overline{f}(\omega,\nu_j)\le\overline f(\omega,\nu)\leq\liminf_{j\to+\infty}\overline{f}(\omega,\nu_j)\,,
\end{equation*}
and the proof of step 2 is achieved.

\medskip

\textit{Step 3: $(\tau_z)_{z\in \Z^n}$-translation invariance.} In this step we show that $\tilde \Omega$ and $f_\hom$ are $(\tau_z)_{z\in \Z^n}$-translation invariant.

Let $z\in\Z^n$, $\omega\in\tilde\Omega$, and $\nu\in\S^{n-1}$ be fixed.  Let $r>4$ \EEE
and $ u\in \Adm(\bar u^\nu_0,Q^\nu_r(0))$ satisfying
\begin{equation}\label{c:quasi-min-s}
\F(\omega)(u,Q_r^\nu(0))\le\m_\omega(\bar u^\nu_0,Q^\nu_r(0))+1\,. 
\end{equation}
Setting $\tilde u(y):=u(y+z)$, then since $f$ is stationary there holds 
\begin{equation*}
\F(\omega)(u,Q_r^\nu(0))=\F(\tau_z(\omega))(\tilde u,Q_r^\nu(-z)).
\end{equation*}
This together with \eqref{c:quasi-min-s} and the fact that $\tilde u\in\Adm(\bar u^\nu_{-z},Q^\nu_r(-z))$ yield 
\begin{equation}\label{shift_prob1}
	\m_{\tau_z(\omega)}(\bar u^\nu_{-z},Q^\nu_r(-z))\le\m_\omega(\bar u^\nu_0,Q^\nu_r(0))+1\,. 
\end{equation}
We choose $r, \tilde r$  such that $\tilde r> r$ and
\begin{equation*}
	Q^\nu_{r+2}(-z)\wcont Q^\nu_{\tilde r}(0)\; \text{ and }\;  
\dist(0,\Pi^\nu(-z))\le\frac r4.
\end{equation*}
We next apply Lemma \ref{lem:cubes-shift} twice: once with $x=-z$ and $\tilde x=0$ to the minimisation problem $\m_{\tau_z(\omega)}$ and once with $x=z$ and $\tilde x=0$ to the minimisation problem $\m_{\omega}$ and get 
\begin{equation}\label{shift_prob2}
    \m_{\tau_z(\omega)}(\bar{u}_{0}^\nu,Q_{\tilde{r}}^\nu(0)) \leq\m_{\tau_z(\omega)}(\bar{u}_{-z}^\nu,Q_{r}^\nu(-z))+L\big(|z|+|r-\tilde{r}|+1\big)(\tilde{r})^{n-2}\,,
\end{equation}
and
\begin{equation}\label{shift_prob3}
        \m_{\omega}(\bar{u}_{0}^\nu,Q_{\tilde{r}}^\nu(0)) \leq\m_{\omega}(\bar{u}_{z}^\nu,Q_{r}^\nu(z))+L\big(|z|+|r-\tilde{r}|+1\big)(\tilde{r})^{n-2}\,.
\end{equation}
Hence, combining \eqref{shift_prob1} with \eqref{shift_prob2}  and \eqref{shift_prob3} we have
\begin{equation}\label{c:uno-ti}
	\frac{\m_{\tau_z(\omega)}(\bar{u}_{0}^\nu,Q_{\tilde{r}}^\nu(0))}{\tilde r^{n-1}}\le
	\frac{\m_{\omega}(\bar u^\nu_0,Q^\nu_r(0))+1}{r^{n-1}}+\frac{L\big(|z|+|r-\tilde{r}|+1\big)}{\tilde r}\,,
\end{equation}
and
\begin{equation}\label{c:due-ti}
\frac{\m_\omega(\bar{u}_{0}^\nu,Q_{\tilde{r}}^\nu(0))}{\tilde r^{n-1}}\le
\frac{\m_{\tau_z(\omega)}(\bar u^\nu_0,Q^\nu_r(0))+1}{r^{n-1}}+\frac{L\big(|z|+|r-\tilde{r}|+1\big)}{\tilde r}\,.
\end{equation}
We now take in \eqref{c:uno-ti} the limsup as $\tilde r\to+\infty$ and the limit as $r\to+\infty$ and find
\begin{equation}\label{c:ti-1}
\limsup_{\tilde r\to+\infty}	\frac{\m_{\tau_z(\omega)}(\bar{u}_{0}^\nu,Q_{\tilde{r}}^\nu(0))}{\tilde r^{n-1}}\le f_\hom(\omega,\nu).
\end{equation}
Similarly we  take in \eqref{c:due-ti} the limit as $\tilde r\to+\infty$ and the liminf as $r\to+\infty$ and obtain
\begin{equation}\label{c:ti-2}
f_\hom(\omega,\nu)\le \liminf_{r\to+\infty}\frac{\m_{\tau_z(\omega)}(\bar{u}_{0}^\nu,Q_{r}^\nu(0))}{r^{n-1}}.
\end{equation}
Gathering \eqref{c:ti-1} and \eqref{c:ti-2} we deduce that
$\tau_z(\omega)\in\tilde\Omega$ and that
\[
f_\hom(\tau_z(\omega),\nu)=f_\hom(\omega,\nu)\,,
\]
for every $z\in \Z$, $\omega\in\tilde{\Omega}$, and $\nu\in\Sph^{n-1}$.
Notice also that thanks to the group properties of $(\tau_z)_{z\in\Z^n}$ we also have that $\om \in \tau_z(\tilde \Omega)$, for every $z\in \Z^n$. 
%Indeed we have $\om = \tau_{z}(\tau_{-z}(\om))$ and $\tau_{-z}(\om) \in \tilde \Omega$. 

We conclude by observing that if $(\tau_z)_{z\in\Z^n}$ is ergodic, then the independence of $\omega$ of the function of $f_\hom$ follows by~\eqref{eq:shift-invariance} (cf.~\cite[Corollary 6.3]{CDMSZ19a}) and \eqref{c:f-hom-er} follows by integrating \eqref{eq:ex-limit-zero} on $\Omega$ and by using the Dominated Convergence Theorem, thanks to \eqref{c:bd-mu} (see also \eqref{def:subad-process-g}).
\end{proof}
We conclude this section by establishing the last crucial result which extends Proposition \ref{prop:ex-limit-zero} to the case of an arbitrary $x\in \R^n$. More precisely, Proposition \ref{prop:ex-limit-x} below states that the limit in \eqref{eq:ex-limit-zero} exists when $x=0$ is replaced by any $x\in \R^n$ and that it is $x$-independent, and hence it coincides with \eqref{eq:ex-limit-zero}. 

The proof of the following proposition can be obtained arguing exactly as in \cite[Theorem 6.1]{CDMSZ19a} (see also~\cite[Theorem 5.5]{ACR11}), now appealing to Proposition \ref{prop:ex-limit-zero}, Lemma~\ref{lem:cubes-shift}, and Lemma~\ref{lem:cubes-rotation}. For this reason we skip its proof here. 

\begin{proposition}[Homogenised surface integrand]\label{prop:ex-limit-x}
Let $f$ satisfy \ref{meas_f}-\ref{cont_in_xi} and assume that it is stationary with respect to a group $(\tau_z)_{z\in\Z^n}$ of $P$-preserving transformations on $(\Omega,\T,P)$. For $\omega\in\Omega$ let $\m_\omega$ be as in~\eqref{m-omega}. Then there exists $\Omega'\in \T$ with $P(\Omega')=1$ such that 
\begin{equation}\label{eq:ex-limit-x}
\lim_{r\to+\infty}\frac{\m_\omega(\bar{u}_{rx}^\nu,Q_{r}^\nu(rx))}{r^{n-1}}=f_\hom(\om,\nu)
\end{equation}
for every $\omega\in\Omega'$, every $x\in \R^n$, and every $\nu\in\Sph^{n-1}$, where $f_\hom$ is given by~\eqref{eq:ex-limit-zero}. In particular, the limit in~\eqref{eq:ex-limit-x} is independent of $x$. Moreover, if $(\tau_z)_{z\in\Z^n}$ is ergodic, then $f_\hom$ is independent of $\omega$ and given by~\eqref{c:f-hom-er}.
\end{proposition}
We conclude this section with the proof of Theorem \ref{thm:stoch_hom_2}.
\begin{proof}[Proof of Theorem \ref{thm:stoch_hom_2}]
The proof follows by Theorem~\ref{thm:hom} now invoking Proposition \ref{prop:ex-limit-x}.
\end{proof}
\section{Periodic Homogenisation}
\label{sect:periodic-hom}In this last section we prove a periodic homogenisation result without requiring any lower semicontinuity  on the integrand $f$ (i.e., without assuming \ref{cont_in_xi}).\\ 

% Notice first that the notion of stationarity  given in Definition \ref{def:stationarity} generalises the notion of spatial periodicity.
% Therefore our stochastic homogenisation result Theorem \ref{thm:stoch_hom_2} applies also to the framework of periodic homogenisation. 
A careful inspection of the proof of Theorem \ref{thm:stoch_hom_2} shows  that condition \ref{cont_in_xi} is used only once in the proof of Proposition \ref{prop:subadditive} to establish the $\T$-measurability of the map $\mu_\nu$ defined in \eqref{def:subad-process-g}.
This suggests that in the periodic setting  Theorem \ref{thm:stoch_hom_2} should still be true if we drop condition \ref{cont_in_xi}.
\\

% The proof of the existence and the homogeneity of the limit in~\eqref{eq:f-stoch-hom},  and the $\mathcal B(\Sph^{n-1})$-measurability of $f_\hom$ in the periodic setting without assuming \ref{cont_in_xi} is rather standard and follows by  suitably adapting the arguments in \cite{AnBraCP}. For this reason here we will omit it here. 
% Let $f\in\mathcal F$ be such that $f(\cdot,u,\xi)$ is $Q$-periodic for every $u\in\R$ and every $\xi\in\R^$ and let $\m(\bar u_{rx}^\nu,Q^\nu_r(rx))$ be as in \eqref{eq:m-bis} with $z=rx$ and $A=Q_r^\nu(rx)$. Then arguing as in \cite{AnBraCP}[Theorem 4.3] it can be shown that the limit
% \begin{equation*}
%     ....
% \end{equation*}
% exists for every $\nu\in \Q^n\cap\S^{n-1}$. Next similarly to step 2 in the proof of Proposition \ref{prop:ex-limit-zero} we can deduce the existence of the limit for every $\nu\in \S^{n-1}$. Eventually ...

Let $f\in\mathcal F$ be such that for every $u\in\R$ and every $\xi\in\R^n$ $f(\cdot,u,\xi)$ is $Q$-periodic. Let $f_k$ and $\F_k$ 
be 
defined as in \eqref{eq:f-osc} and \eqref{F_e} accordingly.
% We consider the sequence of functionals $\F_k \colon  L^1_\loc(\R^n) \times \A \longrightarrow [0,+\infty]$ given by
% \begin{equation}\label{F_e-per}
% \F_k(u, A)\defas \frac{1}{\e_k}\int_A f\left(\frac{x}{\e_k},u,\e_k\nabla u\right)\dx\,,
% \end{equation}
% %
% if $u\in  W^{1,p}(A)$, $0\leq u\leq 1$ and extended to $+\infty$ otherwise.  \\
We now state and prove the following result.
\begin{theorem}[Periodic homogenisation]\label{thm:det_hom_2}
Let $f\in\mathcal{F}$ be such that for every $u\in\R$ and every $\xi\in\R^n$ $f(\cdot,u,\xi)$ is $Q$-periodic.  Let  $\m$ be as in~\eqref{eq:m-bis}.  Then the limit
	\begin{equation}\label{eq:f-per-hom}
	\lim_{r\to +\infty} \frac{\m(\bar u^\nu_{rx},Q^\nu_r(rx))}{r^{n-1}}=
	\lim_{r\to +\infty} \frac{\m(\bar u^\nu_{0},Q^\nu_r(0))}{r^{n-1}}=:f_\hom(\nu)
	\end{equation}
exists and is independent of $x$. The function  $f_\hom \colon  \Sph^{n-1} \to [0,+\infty)$ is $ \mathcal B(\Sph^{n-1})$-measurable.

Moreover, for every $A\in\A$, the functionals $\F_k(\cdot,A)$, defined in \eqref{F_e} with $f_k$ as in \eqref{eq:f-osc}, $\Gamma$-converge in $ L^1_\loc(\R^n)$ to the functional $\F_{\hom}(\cdot,A)$ with $\F_{\hom}\colon  L^1_\loc(\R^n)\times\A\longrightarrow[0,+\infty]$ given by
\begin{equation*}
\F_{\rm hom}(u,A)\defas
\begin{cases}
\displaystyle\int_{S_u \cap A} f_{\rm hom}(\nu_u)\dHn &\text{if}\ u \in BV(A;\{0,1\})\, \\[4pt]
+\infty &\text{otherwise}\,.
\end{cases}
\end{equation*}
\end{theorem}
\begin{proof}
From Theorem \ref{thm:hom} it is sufficient to show that the limit in \eqref{eq:f-per-hom} exists and is independent of $x$ and that $f_\hom$ is $ \mathcal B(\Sph^{n-1})$-measurable. We divide the proof into a nuber of steps.\\

\textit{Step 1: 
existence of the limit for $x=0$. 
}
In this step we show that for $x=0$ and for every $\nu\in \Sph^{n-1} $ the following limit in \eqref{eq:f-per-hom} exists.
Let $\nu\in \Sph^{n-1}$ and let $s>r>0$ be fixed.
By reasoning as in step 2 of the proof of Proposition \ref{prop:ex-limit-zero} we can  assume without loss of generality that $\nu\in\Sph^{n-1}\cap\Q^n$. Next we choose $u_r\in \Adm(\bar u^\nu_0,Q_r^\nu(0))$ such that 
 \begin{equation}\label{quasi-min}
     \F(u_r,Q_r^\nu(0))\le \m(\bar u^\nu_{0},Q^\nu_r(0))+1\,.
 \end{equation}
In what follows we  extend $u_r$ to $Q^\nu_s(0)$ to get a new function $u_s\in \Adm(\bar u^\nu_0,Q_s^\nu(0))$ without essentially increasing the energy. 
For every $z\in\Z^{n-1}\times\{0\}$ we set
\begin{equation*}
z_r^\nu\defas\left(\left\lfloor \frac{r}{M_\nu}\right\rfloor+1\right)M_\nu R_\nu z\,,\quad
    \tilde Q_z\defas Q_r^\nu(0)+z_r^\nu\,,
\end{equation*}
with $R_\nu$ as in \ref{Rn} and $M_\nu$ defined as in section \ref{sect:cell-formulas}, and let
\begin{equation*}
    I_s\defas\left\{z\in\Z^{n-1}\times\{0\}\colon  \tilde Q_z\subset Q_s^\nu(0)\right\}\,.
\end{equation*}
Note that by definition of $M_\nu$ and $R_\nu$ we have $z_r^\nu\in\Z^n\cap\Pi_\nu$ for every $z\in\Z^{n-1}\times\{0\}$.
Moreover a direct computation shows that
\begin{equation}\label{cardinalita}
    s^{n-1}\left(\frac{1}{r+1}-\frac1s\right)^{n-1}
\le \#(I_s)\le \frac{s^{n-1}}{r^{n-1}}\,.
\end{equation}
 Let us define
\begin{equation*}
    u_s(y)\defas\begin{cases}
 u_r(y-z_r^\nu)&\text{ if }y\in \tilde Q_z,\, z\in I_s\,,\\
 \bar u^\nu_{0}(y)&\text{ otherwise in } Q^\nu_s(0)\,.
    \end{cases}
\end{equation*}
Clearly $u_s\in \Adm(\bar u^\nu_0,Q_s^\nu(0))$. Furthermore a change of variable together with the the $Q$-periodicity of $f$ yield
\begin{equation}\label{QsQr}
    \begin{split}
       \F(u_s,Q_s^\nu(0))   \le \#(I_s)  \F(u_r,Q_r^\nu(0)) + \F(\bar u^\nu_{0}, Q_s^\nu(0)\setminus \cup_{z\in I_s}\tilde Q_z)\,.
    \end{split}
\end{equation}
Moreover by \eqref{1dim-energy-bis} 
we may deduce that 
\begin{equation}\label{sec-termine}
     \F(\bar u^\nu_{0}, Q_s^\nu(0)\setminus \cup_{z\in I_s}\tilde Q_z)\le c_2C_\uu(s^{n-1}- \#(I_s)r^{n-1})\,.
\end{equation}
Gathering together \eqref{cardinalita}-\eqref{sec-termine} we obtain
\begin{equation} \label{this}
\begin{split}
     \frac{\m(\bar u^\nu_{0},Q^\nu_s(0))}{s^{n-1}}\le \frac{ \F(u_s,Q_s^\nu(0)) }{s^{n-1}}\le\frac{\F(u_r,Q_r^\nu(0))}{r^{n-1}}    +c_2C_\uu \left(1- r^{n-1}\Big(\frac{1}{r+1}-\frac1s\Big)^{n-1}\right)\,.
\end{split}
\end{equation}
Finally combining \eqref{this} with \eqref{quasi-min} and passing first to the limsup as  $s\to+\infty$ and then to the liminf as $r\to+\infty$ we get
\begin{equation*}
    \limsup_{s\to+\infty} \frac{\m(\bar u^\nu_{0},Q^\nu_s(0))}{s^{n-1}}\le \liminf_{r\to+\infty} \frac{\m(\bar u^\nu_{0},Q^\nu_r(0))}{r^{n-1}}\,,
\end{equation*}
and, being the converse inequality trivial, the proof of Step 1 is achieved.

\textit{Step 2: 
existence of the limit for every $x$.
} In this step we show that for every $x\in\R^n$ the limit in \eqref{eq:f-per-hom} exists and is $x$-independent. Let $x\ne0$, and let $r\ge4\sqrt{n}$ be fixed. Let  $u_r\in \Adm(\bar u^\nu_0,Q_r^\nu(0))$ be such that \eqref{quasi-min} holds. Define $x_r\defas\lfloor (r+3)x\rfloor$ (i.e., the vector whose componets are the integer parts of the components of $(r+3)x$) so that $|(r+3)x-x_r |\le\sqrt{n}$
and moreover
\begin{equation*}
    Q_{r+2}^\nu( x_r)\subset \subset Q_{r+3}^\nu( (r+3)x)\,, \quad\text{ and } \quad
    \dist(  (r+3)x,\Pi^\nu( x_r))\leq\frac{r}{4}\,.
\end{equation*}
The function  $\hat u_r(y)\defas u_r(y- x_r)$ belongs to $\Adm(\bar u^\nu_{x_r},Q_r^\nu( x_r))$. Moreover by $Q$-periodicity of $f$ and \eqref{quasi-min} we have 
\begin{equation} \label{comb1}
\m(\bar u^\nu_{ x_r},Q^\nu_r( x_r))\le
\F(\hat u_r,Q_r^\nu(x_r))=
    \F(u_r,Q_r^\nu(0))\le \m(\bar u^\nu_{0},Q^\nu_r(0))+1\,.
\end{equation}
Next we invoke Lemma \ref{lem:cubes-shift} and get 
\begin{equation}\label{comb2}
    \m(\bar u^\nu_{ (r+3)x},Q^\nu_{r+3}((r+3) x))\le \m(\bar u^\nu_{ x_r},Q^\nu_r( x_r))+ L
    \big(\sqrt{n}+3+1\big)(r+3)^{n-2}\,.
\end{equation}
Combining \eqref{comb1} with \eqref{comb2} and rescaling by $(r+3)^{n-1}$ we have
\begin{equation*}
    \frac{\m(\bar u^\nu_{(r+3) x},Q^\nu_{r+3}((r+3) x))}{(r+3)^{n-1}}\le \frac{r^{n-1}}{(r+3)^{n-1}}\frac{ \m(\bar u^\nu_{0},Q^\nu_r(0))}{r^{n-1}}+ \frac{C}{r+3}\,,
\end{equation*}
hence, up to replacing $r+3$ with $r$, passing to the limsup as $r\to+\infty$ we find 
\begin{equation*}
    \limsup_{r\to+\infty}  \frac{\m(\bar u^\nu_{r x},Q^\nu_{r}(r x))}{r^{n-1}}\le
     \limsup_{r\to+\infty}  \frac{ \m(\bar u^\nu_{0},Q^\nu_r(0))}{r^{n-1}}=\lim_{r\to+\infty}  \frac{ \m(\bar u^\nu_{0},Q^\nu_r(0))}{r^{n-1}}\,.
\end{equation*}
In a similar way one can show that also 
\begin{equation*}
     \lim_{r\to+\infty}  \frac{ \m(\bar u^\nu_{0},Q^\nu_r(0))}{r^{n-1}}=\liminf_{r\to+\infty}  \frac{ \m(\bar u^\nu_{0},Q^\nu_r(0))}{r^{n-1}}\le   \liminf_{r\to+\infty}  \frac{\m(\bar u^\nu_{ rx},Q^\nu_{r}( rx))}{r^{n-1}}
     \,.
\end{equation*}
The above two inequalities conclude step 2.

\textit{Conclusions.
} From step 2 we deduce \eqref{eq:f-per-hom}. Moreover arguing as in the step 2 of the proof of Proposition \ref{prop:ex-limit-zero} we may also deduce that $f_\hom$ is $ \mathcal B(\Sph^{n-1})$-measurable.

\end{proof}
\begin{remark}
	We conclude this section by observing that our analysis in the stochastic/periodic setting (and the corresponding results Theorem \ref{thm:stoch_hom_2} and Theorem \ref{thm:det_hom_2}) covers, in particular, the case of functionals considered in \cite{Morfe, CFHP} and the critical regime of \cite{AnBraCP}.
\end{remark}

\section*{Acknowledgments}
\noindent 
The author wishes to thank Caterina Ida Zeppieri for having drawn her attention to this problem and for many helpful discussions and suggestions, and Annika Bach for useful comments on a preliminary version.
The present work was supported by the Deutsche Forschungsgemeinschaft (DFG, German Research Foundation) project 3160408400 and under the Germany Excellence Strategy EXC 2044-390685587, Mathematics M\"unster: Dynamics--Geometry--Structure.

\appendix
\section*{Appendix}\label{appendix}
\setcounter{theorem}{0}
\setcounter{equation}{0}
\renewcommand{\theequation}{A.\arabic{equation}}
\renewcommand{\thetheorem}{A.\arabic{theorem}}

\noindent In this last section we state and prove two technical lemmas which are used in Subsection \ref{sect:cell-formulas}.

\medskip

For $A \in \mathcal A$, $x\in \R^n$, and $\nu \in \Sph^{n-1}$, in what follows $\m(\bar u^\nu_x, A)$ denotes the infimum value given by \eqref{eq:m-bis}. 

\begin{lemma}\label{lem:cubes-shift}
Let $f \in \mathcal F$; let $\nu\in\Sph^{n-1}$, $x,\tilde{x}\in\R^n$, and $\tilde{r}> r>4$ be such that
\begin{equation*}%\label{cond:cubes-shift}
{\rm (i)}\ Q_{r+2}^\nu(x)\wcont Q_{\tilde{r}}^\nu(\tilde x)\,,\qquad {\rm (ii)}\ \dist(\tilde x,\Pi^\nu(x))\leq\frac{r}{4}\,.
\end{equation*}
Then there exists a constant $L>0$ (independent of $\nu,x,\tilde{x},r,\tilde{r}$) such that
\begin{equation}\label{eq:cubes-shift}
\m(\bar{u}_{\tilde{x}}^\nu,Q_{\tilde{r}}^\nu(\tilde x))\leq\m(\bar{u}_{x}^\nu,Q_{r}^\nu(x))+L\big(|x-\tilde x|+|r-\tilde{r}|+1\big)\tilde{r}^{n-2}\,.
\end{equation}
\end{lemma}
\begin{proof}

Let $\nu\in S^{n-1}$, $\eta>0$ be fixed and let $ u\in\Adm(\bar{u}^\nu_x,Q^\nu_r(x))$  with
\begin{equation}\label{est:energy-shift}
\F(u,Q^\nu_{r}(x))\leq\m(\bar{u}^\nu_{x},Q^\nu_{r}(x))+\eta\,.
\end{equation}
Then $u=\bar{u}^\nu_{x}$ a.e.\ in $U$, where $U$ is a neighbourhood of $\partial Q^\nu_{r}(x)$. Let moreover $\beta\in(0,1)$ be such that $Q^\nu_{r}(x)\sm\overline{Q}^\nu_{r-\beta}(x)\subset U$. Let $\tilde x \in \R^n$ and $\tilde r>r>4$ satisfy (i)-(ii). We set
\begin{equation*}
R\defas R_\nu\biggl( \big(Q'_{r}\sm\overline{Q'}_{r-\beta}\big)\x\Big(-1-\frac{|(x- \tilde{x})\cdot\nu|}{2}\,,1+\frac{|(x-\tilde{x})\cdot\nu|}{2}\Big)\biggr)+x+\frac{(\tilde x-x)\cdot\nu}{2} \nu\,,
\end{equation*}
where $R_\nu$ be as in \ref{Rn}. Let $\varphi$ be a smooth cutoff between $Q^\nu_{r-\beta}(x)$ and $Q^\nu_{\tilde r}(\tilde x)$ and define $$\tilde u\defas \varphi u+(1-\varphi)\bar u^\nu_{\tilde x}\,,$$
which clearly belongs to $ \Adm(\bar u^\nu_{\tilde x},Q^\nu_{\tilde r}(\tilde x))$ and moreover it satisfies
\begin{equation}\label{fund-est-bdry}
\F(\tilde u,Q^\nu_{\tilde r}(\tilde x))\le 
\F(u,Q_r^\nu(x))+\F(\tilde u,Q^\nu_{\tilde r}(\tilde x)\setminus \overline Q^\nu_{r-\beta}(\tilde x))
\,.
\end{equation}
Notice that
$$ \tilde u= \varphi  \bar u^\nu_{ x}+(1-\varphi)\bar u^\nu_{\tilde x}\quad \text{ in }\quad Q^\nu_{\tilde r}(\tilde x)\setminus \overline{Q}^\nu_{r-\beta}(x)\,,$$
and in particular 
$$\tilde u={u}^\nu_{x}= u^\nu_{\tilde x}\in\{0,1\}\quad \text{ a.e. in }\quad (Q^\nu_{\tilde r}(\tilde x)\setminus \overline{Q}^\nu_{r-\beta}(x))\setminus \overline R\,.$$
This together with~\eqref{f-value-at-0} and \ref{hyp:ub-f} imply
\begin{equation}\label{c:c0}
\begin{split}
  \F(\tilde u,Q^\nu_{\tilde r}(\tilde x)\setminus \overline Q^\nu_{r-\beta}(\tilde x))&\le c_2\int_{R}\left(W(\tilde u)+|\nabla\tilde u|^p\right)\dy \\
  &\le  C\L^n(R)+ C\int_{R}\left( |\nabla \bar u^\nu_{ x}|^p+|\nabla \bar u^\nu_{\tilde x}|^p+ |\nabla\varphi|^p|\bar u^\nu_{\tilde x}-\bar u^\nu_{ x}|^p
  \right)\dy\\
  &\le C\L^n(R)+ C \L^{n-1}(Q'_r\setminus\overline Q'_{r-\beta})\,, 
\end{split}
\end{equation}
where the last two inequality follow by recalling that $W$ is bounded on compact sets and by \eqref{1dim-energy-bis}.
Now we observe that
\begin{equation}\label{c:c1}
	\mathcal{L}^{n-1}\big(Q'_{\tilde{r}}\sm\overline{Q}'_{r-\beta}\big)\leq C|r-\tilde{r}|\tilde{r}^{n-2}\,,
\end{equation}
while
\begin{equation}\label{c:c2}
\L^n(R)\le C\beta r^{n-2}(|x-\tilde x|+1)\,.
\end{equation}
Finally gathering~\eqref{est:energy-shift}, \eqref{fund-est-bdry}, \eqref{c:c0}, \eqref{c:c1} and \eqref{c:c2}  we obtain
\begin{equation*}
\m(\bar{u}^\nu_{\tilde x},Q^\nu_{\tilde{r}}(\tilde x))\leq\F\big(\tilde{u},Q^\nu_{\tilde{r}}(\tilde x)\big)\leq\m(\bar{u}^\nu_{x},Q^\nu_r(x))+L\big(|x-\tilde x|+|r-\tilde{r}|+1\big)\tilde{r}^{n-2}+\eta\,,
\end{equation*}
for some $L>0$ independent of $x,\tilde x, r, \tilde r, \nu$, thus \eqref{eq:cubes-shift} follows by the arbitrariness of $\eta>0$. 
\end{proof}
\begin{lemma}\label{lem:cubes-rotation}
Let $f \in \mathcal F$; let $\alpha\in (0,\frac{1}{2})$ and $\nu,\tilde{\nu}\in\Sph^{n-1}$ be such that 
\begin{equation}\label{c:30-condition}
\max_{1\leq i\leq n-1}|R_\nu e_i-R_{\tilde \nu} e_i| + |\nu -\tilde \nu| <\frac{\alpha}{\sqrt n},
\end{equation}
where $R_\nu$ and $R_{\tilde \nu}$ are orthogonal $(n\times n)$-matrices as in \ref{Rn}. 
Then there exists a constant $c_\alpha>0$ (independent of $\nu,\tilde{\nu}$), with $c_\alpha\to 0$ as $\alpha\to 0$, such that for every $x\in\R^n$ and every $r>2$ we have
\begin{equation}\label{est:cubes-rotation}
\begin{split}
\m\big(\bar{u}_{rx}^{\tilde{\nu}},Q_{(1+\alpha)r}^{\tilde{\nu}}(rx)\big)-c_\alpha r^{n-1} &\leq \m\big(\bar{u}_{rx}^\nu, Q_{r}^\nu(rx)\big)\\
&\leq \m\big(\bar{u}_{rx}^{\tilde{\nu}},Q_{(1-\alpha)r}^{\tilde{\nu}}(rx)\big)+c_\alpha r^{n-1}\,.
\end{split}
\end{equation}
\end{lemma}
\begin{proof}
We show only that 
\begin{equation}\label{c:30-claim}
\m\big(\bar{u}_{rx}^{\tilde{\nu}},Q_{(1+\alpha)r}^{\tilde{\nu}}(rx)\big)-c_\alpha r^{n-1} \leq \m\big(\bar{u}_{rx}^\nu, Q_{r}^\nu(rx)\big),
\end{equation}
for some $c_\alpha>0$, with $c_\alpha \to 0$, as $\alpha \to 0$; as the proof of the other inequality is analogous.
Let $x\in\R^n$, $r>2$, and set $r_\alpha^\pm\defas (1\pm\alpha)r$. From \eqref{c:30-condition} we might deduce that 
\begin{equation*}\label{cond:cube-rotation-center-x} 
Q_{ r_\alpha^-}^{\tilde{\nu}}(rx) \wcont Q_r^\nu(rx)\wcont Q_{ r_\alpha^+}^{\tilde{\nu}}(rx).
\end{equation*}
Let $\eta>0$ be fixed and let $u\in W^{1,p}(Q^\nu_r(rx))$ be a test function for $\m(\bar u_{rx}^\nu,Q^\nu_r(rx))$ satisfying 
\begin{equation}\label{cont_v}
\F(u,Q^\nu_r(rx))\le \m(\bar u_{rx}^\nu,Q^\nu_r(rx))+\eta.
\end{equation}
Choose $\beta\in(0,1)$ such that $Q_{ r_\alpha^-}^{\tilde{\nu}}(rx) \wcont Q^\nu_{r-\beta}(rx)$ and $u=\bar u^\nu_{rx}$ a.e. in $Q_r^\nu(rx)\setminus \overline{Q}_{r-\beta}^\nu(rx)$ and set
\begin{equation*}
R:=R_\nu\Big(Q'_{r} \setminus\overline {Q'}_{r-\beta}\x(-1-\alpha r,1+\alpha r)\Big)+rx\,,
\end{equation*}
where $R_\nu$ is as in~\ref{Rn}. 
Next define 
$$\tilde u\defas \varphi u+(1-\varphi)\bar{u}^{\tilde\nu}_{rx}\,,$$
where $\varphi$ is a smooth cutoff between $Q^\nu_{r-\beta}(rx)$ and $Q^{\nu}_{r}(rx)$. Then $\tilde u\in \Adm(\bar u^{\tilde\nu}_{r x},Q_{ r_\alpha^+}^{\tilde{\nu}}(rx))$ and it also satisfies
\begin{equation}\label{fund-est-bdry1}
	\F(\tilde u,Q^{\tilde\nu}_{r_\alpha^+}(rx))\le 
	\F(u,Q_r^\nu(rx))+ 	\F(\tilde{u} ,
Q^{\nu}_{r}(rx)\setminus \overline{Q}_{r-\beta}^\nu(rx))+
	\F(\bar{u}^{\tilde\nu}_{rx} ,
Q^{\tilde\nu}_{r_\alpha^+}(rx)\setminus \overline{Q}_{r_\alpha^-}^\nu(rx))
	\,.
\end{equation}
By definition we have 
$$ \tilde u= \varphi \bar u^\nu_{rx}+(1-\varphi)\bar{u}^{\tilde\nu}_{rx}\quad\text{ in }\quad Q^{\nu}_{r}(rx)\setminus \overline{Q}_{r-\beta}^\nu(rx)\,, $$
and in particular 
$$ \tilde u= u^\nu_{rx}={u}^{\tilde\nu}_{rx}\in\{0,1\}\quad \text{ a.e. in }\quad(Q^{\nu}_{r}(rx)\setminus \overline{Q}_{r-\beta}^\nu(rx))\setminus\overline R\,.$$
Therefore from~\eqref{f-value-at-0}, \eqref{hyp:ub-f} and \eqref{1dim-energy-bis} we deduce that
\begin{equation}\label{c:b0}
\begin{split}
\F(\tilde{u} ,
Q^{\nu}_{r}(rx)\setminus \overline{Q}_{r-\beta}^\nu(rx))
 &\le c_2\int_{R}\left(W(\tilde u)+|\nabla\tilde u|^p\right)\dy \\
  &\le  C\L^n(R)+ C\int_{R}\left( |\nabla \bar u^\nu_{r x}|^p+|\nabla \bar u^{\tilde\nu}_{r x}|^p+ |\nabla\varphi|^p|\bar u^\nu_{r x}-\bar u^{\tilde\nu}_{ rx}|^p
  \right)\dy\\
  &\le C\L^n(R)+ C \L^{n-1}(Q'_r\setminus\overline Q'_{r-\beta})\,.
\end{split}
\end{equation}
On the other hand from \eqref{1dim-energy-bis} and \eqref{f-value-at-0} we infer 
\begin{equation}\label{1}
\F(\bar{u}^{\tilde\nu}_{rx}(y) ,
Q^{\tilde\nu}_{r_\alpha^+}(rx)\setminus Q^{\tilde\nu}_{r_\alpha^-}(rx))\le c_2C_\uu ((1+\alpha)^{n-1}-(1-\alpha)^{n-1})r^{n-1}\,.
\end{equation}
Using that 
\begin{equation*}
\L^{n-1}(Q'_r\setminus\overline Q'_{r-\beta})\le Cr^{n-2}
 \,,\quad\text{ and }\quad \L^n(R)\le C\beta r^{n-2}\alpha r\le C\alpha r^{n-1}\,.
\end{equation*}
together with \eqref{fund-est-bdry1}, \eqref{c:b0}, and \eqref{1} we obtain
\begin{equation*}\label{fund-est-bdry1-bis}
	\F(\tilde u,Q^{\tilde\nu}_{r_\alpha^+}(rx))\le 
	\F(u,Q_r^\nu(rx))+\F(\bar{u}^{\tilde\nu}_{rx}(y) ,
Q^{\tilde\nu}_{r_\alpha^+}(rx)\setminus Q^{\tilde\nu}_{r_\alpha^-}(rx)+c_\alpha r^{n-1}
\,,
\end{equation*}
where $c_\alpha\defas c_4C_{\uu}\big((1+\alpha)^{n-1}-(1-\alpha)^{n-1}\big)+C\alpha$.
Eventually as the above inequality together with \eqref{cont_v} imply
\begin{align*}
\m(\bar{u}_{ rx}^{\tilde{\nu}},Q_{ r_\alpha^+}^{\tilde{\nu}}(rx)) \leq\m(\bar{u}_{rx}^\nu,Q_r^\nu(rx))+c_\alpha r^{n-1} +\eta\,,
\end{align*} 
 we  may deduce \eqref{c:30-claim} by the arbitrariness of $\eta>0$.
\end{proof}
\EEE

\end{document}